\documentclass[reqno]{amsart}
\usepackage{amsmath}
\usepackage{amsthm}
\usepackage{amssymb}
\usepackage{bbm}
\usepackage{lineno}

\input xy
\xyoption{all}

\numberwithin{equation}{section}

\usepackage{hyperref}

\newcommand{\inacc}{{\rm inacc}}

\newcommand{\GCH}{{\rm GCH}}

\newcommand{\ZFC}{{\rm ZFC}}

\renewcommand{\emptyset}{\varnothing}

\newcommand{\bb}{\mathbb}
\renewcommand{\P}{{\mathbb P}}

\newcommand{\Q}{{\mathbb Q}}
\newcommand{\R}{{\mathbb R}}
\renewcommand{\S}{\mathbb{S}}

\newcommand{\F}{{\mathbb F}}

\newcommand{\T}{{\mathbb T}}

\newcommand{\forces}{\Vdash}

\newcommand{\restrict}{\upharpoonright}
\newcommand{\concat}{\mathbin{{}^\smallfrown}}
\newcommand{\la}{\langle}
\newcommand{\ra}{\rangle}
\newcommand{\<}{\langle}
\renewcommand{\>}{\rangle}

\newcommand{\st}{:}

\newcommand{\defn}{\mathop{\rm def}}

\newcommand{\dom}{\mathop{\rm dom}}

\newcommand{\cf}{\mathop{\rm cf}}
\newcommand{\cof}{\mathop{\rm cof}}

\newcommand{\crit}{\mathop{\rm crit}}

\newcommand{\Tr}{{\mathop{\rm Tr}}}

\newcommand{\Refl}{{\mathop{\rm Refl}}}

\newcommand{\Levy}{L\'{e}vy}

\renewcommand{\and}{\mathop{\&}}

\newcommand{\lt}{{\smalllt}}
\newcommand{\smalllt}{\mathrel{\mathchoice{\raise2pt\hbox{$\scriptstyle<$}}{\raise1pt\hbox{$\scriptstyle<$}}{\raise0pt\hbox{$\scriptscriptstyle<$}}{\scriptscriptstyle<}}}
\newcommand{\smallleq}{\mathrel{\mathchoice{\raise2pt\hbox{$\scriptstyle\leq$}}{\raise1pt\hbox{$\scriptstyle\leq$}}{\raise1pt\hbox{$\scriptscriptstyle\leq$}}{\scriptscriptstyle\leq}}}
\newcommand{\lesseq}{{\smallleq}}
\newcommand{\one}{\mathbbm{1}} 
\newcommand{\image}{\mathbin{\hbox{\tt\char'42}}}
\newcommand{\from}{\mathbin{\vbox{\baselineskip=2pt\lineskiplimit=0pt
                         \hbox{.}\hbox{.}\hbox{.}}}}

\newtheorem{theorem}{Theorem}[section]
\newtheorem{lemma}[theorem]{Lemma}
\newtheorem{corollary}[theorem]{Corollary}
\newtheorem{proposition}[theorem]{Proposition}
\newtheorem{claim}[theorem]{Claim}

\newtheorem*{theorem_1_square_wc}{Theorem \ref{theorem_1_square_wc}}
\newtheorem*{theorem_1_square_and_reflection}{Theorem \ref{theorem_1_square_and_reflection}}

\theoremstyle{definition}
\newtheorem{question}[theorem]{Question}
\newtheorem{remark}[theorem]{Remark}

\newtheorem{definition}[theorem]{Definition}

\date{\today}

\begin{document}


\title{Forcing a $\square(\kappa)$-like principle to hold at a weakly compact cardinal}

\author[Brent Cody]{Brent Cody}
\address[Brent Cody]{
Virginia Commonwealth University,
Department of Mathematics and Applied Mathematics,
1015 Floyd Avenue, PO Box 842014, Richmond, Virginia 23284, United States
}
\email[B. ~Cody]{bmcody@vcu.edu}
\urladdr{http://www.people.vcu.edu/~bmcody/}

\author[Victoria Gitman]{Victoria Gitman}
\address[Victoria Gitman]{
The City University of New York, CUNY Graduate Center, Mathematics Program, 365 Fifth Avenue, New York, NY 10016, USA
}
\email[V. ~Gitman]{vgitman@nylogic.org}
\urladdr{https://victoriagitman.github.io/}

\author[Chris Lambie-Hanson]{Chris Lambie-Hanson}
\address[Chris Lambie-Hanson]{
Virginia Commonwealth University,
Department of Mathematics and Applied Mathematics,
1015 Floyd Avenue, PO Box 842014, Richmond, Virginia 23284, United States
}
\email[C. ~Lambie-Hanson]{cblambiehanso@vcu.edu}
\urladdr{http://people.vcu.edu/~cblambiehanso}

\begin{abstract}
Hellsten \cite{MR2026390} proved that when $\kappa$ is $\Pi^1_n$-indescribable, the \emph{$n$-club} subsets of $\kappa$ provide a filter base for the $\Pi^1_n$-indescribability ideal, and hence can also be used to give a characterization of $\Pi^1_n$-indescribable sets which resembles the definition of stationarity: a set $S\subseteq\kappa$ is $\Pi^1_n$-indescribable if and only if $S\cap C\neq\emptyset$ for every $n$-club $C\subseteq\kappa$. By replacing clubs with $n$-clubs in the definition of $\Box(\kappa)$, one obtains a $\Box(\kappa)$-like principle $\Box_n(\kappa)$, a version of which was first considered by Brickhill and Welch \cite{BrickhillWelch}. The principle $\Box_n(\kappa)$ is consistent with the $\Pi^1_n$-indescribability of $\kappa$ but inconsistent with the $\Pi^1_{n+1}$-indescribability of $\kappa$. By generalizing the standard forcing to add a $\Box(\kappa)$-sequence, we show that if $\kappa$ is $\kappa^+$-weakly compact and $\GCH$ holds then there is a cofinality-preserving forcing extension in which $\kappa$ remains $\kappa^+$-weakly compact and $\Box_1(\kappa)$ holds. If $\kappa$ is $\Pi^1_2$-indescribable and $\GCH$ holds then there is a cofinality-preserving forcing extension in which $\kappa$ is $\kappa^+$-weakly compact, $\Box_1(\kappa)$ holds and every weakly compact subset of $\kappa$ has a weakly compact proper initial segment. As an application, we prove that, relative to a $\Pi^1_2$-indescribable cardinal, it is consistent that $\kappa$ is $\kappa^+$-weakly compact, every weakly compact subset of $\kappa$ has a weakly compact proper initial segment, and there exist two weakly compact subsets $S^0$ and $S^1$ of $\kappa$ such that there is no $\beta<\kappa$ for which both $S^0\cap\beta$ and $S^1\cap\beta$ are weakly compact.
\end{abstract}



\maketitle




\section{Introduction}
In this paper, we investigate an incompactness principle $\square_1(\kappa)$, which is closely related to $\square(\kappa)$ but is consistent with weak compactness. Let us begin by recalling the basic facts about $\square(\kappa)$.

The principle $\square(\kappa)$ asserts that there is a $\kappa$-length coherent sequence of clubs $\vec{C}=\<C_\alpha\st\alpha\in\lim(\kappa)\>$ that cannot be threaded. For an uncountable cardinal $\kappa$, a sequence $\vec{C}=\<C_\alpha\st\alpha\in\lim(\kappa)\>$ of clubs $C_\alpha\subseteq\alpha$ is called \emph{coherent} if whenever $\beta$ is a limit point of $C_\alpha$ we have $C_\beta=C_\alpha\cap\beta$. Given a coherent sequence $\vec{C}$, we say that $C$ is a \emph{thread} through $\vec{C}$ if $C$ is a club subset of $\kappa$ and $C\cap\alpha=C_\alpha$ for every limit point $\alpha$ of $C$. A coherent sequence $\vec{C}$ is called a $\square(\kappa)$-sequence if it cannot be threaded, and $\square(\kappa)$ holds if there is a $\square(\kappa)$-sequence. It is easy to see that $\square(\kappa)$ implies that $\kappa$ is not weakly compact, and thus $\square(\kappa)$ can be viewed as asserting that $\kappa$ exhibits a certain amount of incompactness. The principle $\square(\kappa)$ was isolated by Todor\v{c}evi\'{c} \cite{MR908147}, building on work of Jensen \cite{MR0309729}, who showed that, if $V = L$, then $\square(\kappa)$ holds for every regular uncountable $\kappa$ that is not weakly compact.

The natural ${\leq}\kappa$-strategically closed forcing to add a $\square(\kappa)$-sequence \cite[Lemma 35]{MR3129734} preserves the inaccessibility as well as the Mahloness of $\kappa$, but kills the weak compactness of $\kappa$ and indeed adds a non-reflecting stationary set. However, if $\kappa$ is weakly compact, there is a forcing \cite{MR3730566} which adds a $\square(\kappa)$-sequence and also preserves the fact that every stationary subset of $\kappa$ reflects. Thus, relative to the existence of a weakly compact cardinal, $\square(\kappa)$ is consistent with $\Refl(\kappa)$, the principle that every stationary set reflects. However, $\square(\kappa)$ implies the failure of the simultaneous stationary reflection principle $\Refl(\kappa,2)$ which states that if $S$ and $T$ are any two stationary subsets of $\kappa$, then there is some $\alpha<\kappa$ with $\cf(\alpha)>\omega$ such that $S\cap \alpha$ and $T\cap\alpha$ are both stationary in $\alpha$. In fact, $\square(\kappa)$ implies that every stationary subset of $\kappa$ can be partitioned into two stationary sets that do not simultaneously reflect \cite[Theorem 2.1]{MR3730566}.

If $\kappa$ is a weakly compact cardinal, then the collection of non--$\Pi^1_1$-indescribable subsets of $\kappa$ forms a natural normal ideal called the $\Pi^1_1$-indescribability ideal:
\[\Pi^1_1(\kappa)=\{X\subseteq\kappa\st \text{$X$ is not $\Pi^1_1$-indescribable}\}.\]
A set $S\subseteq\kappa$ is \emph{$\Pi^1_1$-indescribable} if for every $A\subseteq V_\kappa$ and every $\Pi^1_1$-sentence $\varphi$, whenever $(V_\kappa,\in,A)\models\varphi$ there is an $\alpha\in S$ such that $(V_\alpha,\in,A\cap V_\alpha)\models\varphi$. More generally, a $\Pi^1_n$-indescribable cardinal $\kappa$ carries the analogously defined $\Pi^1_n$-indescribability ideal. It is natural to ask the question: which results concerning the nonstationary ideal can be generalized to the various ideals associated to large cardinals, such as the $\Pi^1_n$-indescribability ideals? It follows from the work of Sun \cite{MR1245524} and Hellsten \cite{MR2026390} that when $\kappa$ is $\Pi^1_n$-indescribable the collection of \emph{$n$-club} subsets of $\kappa$ (see the next section for definitions) is a filter-base for the filter $\Pi^1_n(\kappa)^*$ dual to the $\Pi^1_n$-indescribability ideal, yielding a characterization of $\Pi^1_n$-indescribable sets that resembles the definition of stationarity: when $\kappa$ is $\Pi^1_n$-indescribable, a set $S\subseteq\kappa$ is $\Pi^1_n$-indescribable if and only if $S\cap C\neq\emptyset$ for every $n$-club $C\subseteq\kappa$. Several recent results have used this characterization (\cite{MR2252250}, \cite{MR2653962}, \cite{MR3985624} and \cite{MR4050036}) to generalize theorems concerning the nonstationary ideal to the $\Pi^1_1$-indescribability ideal. For technical reasons discussed below in Section \ref{section_questions}, there has been less success with the $\Pi^1_n$-indescribability ideals for $n>1$. In this article we continue this line of research: by replacing ``clubs'' with ``$1$-clubs'' we obtain a $\square(\kappa)$-like principle $\square_1(\kappa)$ (see Definition~\ref{definition_n_square}) that is consistent with weak compactness but not with $\Pi^1_2$-indescribability. Brickhill and Welch \cite{BrickhillWelch} showed that a slightly different version of $\square_1(\kappa)$, which they call $\square^1(\kappa)$, can hold at a weakly compact cardinal in $L$. See Remark \ref{remark_relationship} for a discussion of the relationship between $\square^1(\kappa)$ and $\square_1(\kappa)$. In this article we consider the extent to which principles such as $\square_1(\kappa)$ can be forced to hold at large cardinals.

We will see that the principle $\square_1(\kappa)$ holds trivially at weakly compact cardinals $\kappa$ below which stationary reflection fails. (This is analogous to the fact that $\square(\kappa)$ holds trivially for every $\kappa$ of cofinality $\omega_1$.) Thus, the task at hand is not just to force $\square_1(\kappa)$ to hold at a weakly compact cardinal, but to show that one can force $\square_1(\kappa)$ to hold at a weakly compact cardinal $\kappa$ even when stationary reflection holds at many cardinals below $\kappa$, so that nontrivial coherence of the sequence is obtained. Recall that when $\kappa$ is $\kappa^+$-weakly compact, the set of weakly compact cardinals below $\kappa$ is weakly compact and much more, so, in particular, the set of inaccessible $\alpha<\kappa$ at which stationary reflection holds is weakly compact. By \cite[Theorem 3.24]{BrickhillWelch}, assuming $V=L$, if $\kappa$ is $\kappa^+$-weakly compact and $\kappa$ is not $\Pi^1_2$-indescribable then $\square_1(\kappa)$ holds. We show that the same can be forced.

\begin{theorem}\label{theorem_1_square_wc}
If $\kappa$ is $\kappa^+$-weakly compact and the $\GCH$ holds, then there is a cofinality-preserving forcing extension in which
\begin{enumerate}
\item $\kappa$ remains $\kappa^+$-weakly compact and
\item $\square_1(\kappa)$ holds.
\end{enumerate}
\end{theorem}

We will also investigate the relationship between $\square_1(\kappa)$ and weakly compact reflection principles. The \emph{weakly compact reflection principle} $\Refl_1(\kappa)$ states that $\kappa$ is weakly compact and for every weakly compact $S\subseteq \kappa$ there is an $\alpha<\kappa$ such that $S\cap\alpha$ is weakly compact. It is straightforward to see that if $\kappa$ is $\Pi^1_2$-indescribable, then $\Refl_1(\kappa)$ holds, and if $\Refl_1(\kappa)$ holds, then $\kappa$ is $\omega$-weakly compact (see \cite[Section 2]{MR3985624}). However, the following results show that neither of these implications can be reversed. The first author \cite{MR3985624} showed that if $\Refl_1(\kappa)$ holds then there is a forcing which adds a non-reflecting weakly compact subset of $\kappa$ and preserves the $\omega$-weak compactness of $\kappa$, hence the $\omega$-weak compactness of $\kappa$ does not imply $\Refl_1(\kappa)$. The first author and Hiroshi Sakai \cite{MR4050036} showed that $\Refl_1(\kappa)$ can hold at the least $\omega$-weakly compact cardinal, and hence $\Refl_1(\kappa)$ does not imply the $\Pi^1_2$-indescribability of $\kappa$. Just as $\square(\kappa)$ and $\Refl(\kappa)$ can hold simultaneously relative to a weakly compact cardinal, we will prove that $\square_1(\kappa)$ and $\Refl_1(\kappa)$ can hold simultaneously relative to a $\Pi^1_2$-indescribable cardinal; this provides a new consistency result which does not follow from the results in \cite{BrickhillWelch} due to the fact that,
if $V = L$, then $\Refl_1(\kappa)$ holds at a weakly compact cardinal if and only if
$\kappa$ is $\Pi^1_2$-indescribable.

\begin{theorem}\label{theorem_1_square_and_reflection}
Suppose that $\kappa$ is $\Pi^1_2$-indescribable and the $\GCH$ holds. Then there is a cofinality-preserving forcing extension in which
\begin{enumerate}
\item $\square_1(\kappa)$ holds,
\item $\Refl_1(\kappa)$ holds and
\item $\kappa$ is $\kappa^+$-weakly compact.
\end{enumerate}
\end{theorem}

In Section \ref{section_preliminaries}, using $n$-club subsets of $\kappa$, we formulate a generalization of $\square_1(\kappa)$ to higher degrees of indescribability. It is easily seen that $\square_n(\kappa)$ implies that $\kappa$ is not $\Pi^1_{n+1}$-indescribable (see Proposition \ref{proposition_n_square_denies_n_plus_1_indescribability} below). However, for technical reasons outlined in Section \ref{section_questions}, our methods do not seem to show that $\square_n(\kappa)$ can hold nontrivially (see Definition \ref{definition_n_square_holds_trivially}) when $\kappa$ is $\Pi^1_n$-indescribable. Our methods do allow for a generalization of Hellsten's $1$-club shooting forcing to $n$-club shooting, and we also show that, if $S$ is a $\Pi^1_n$-indescribable set, a $1$-club can be shot through
$S$ while preserving the $\Pi^1_n$-indescribability of all $\Pi^1_n$-indescribable subsets of $S$.

Finally, we consider the influence of $\square_n(\kappa)$ on \emph{simultaneous} reflection of
$\Pi^1_n$-indescribable sets. We let $\Refl_n(\kappa,\mu)$ denote the following simultaneous reflection principle: $\kappa$ is $\Pi^1_n$-indescribable and whenever $\{S_\alpha\st\alpha<\mu\}$ is a collection of $\Pi^1_n$-indescribable sets, there is a $\beta<\kappa$ such that $S_\alpha\cap\beta$ is $\Pi^1_n$-indescribable for all $\alpha < \mu$. In Section \ref{section_applications}, we show that for $n\geq 1$, if $\square_n(\kappa)$ holds at a $\Pi^1_n$-indescribable cardinal, then the simultaneous reflection principle $\Refl_n(\kappa,2)$ fails (see Theorem \ref{theorem_n_square_refutes_simultaneous_refl}). As a consequence, we show that relative to a $\Pi^1_2$-indescribable cardinal, it is consistent that $\Refl_1(\kappa)$ holds and $\Refl_1(\kappa,2)$ fails (see Corollary \ref{corollary_reflection}).

\section{The principles $\square_n(\kappa)$}\label{section_preliminaries}
Suppose that $\kappa$ is a cardinal. A set $S\subseteq\kappa$ is \emph{$\Pi^1_n$-indescribable} if for every $A\subseteq V_\kappa$ and every $\Pi^1_n$-sentence $\varphi$, whenever $(V_\kappa,\in,A)\models\varphi$ there is an $\alpha\in S$ such that $(V_\alpha,\in,A\cap V_\alpha)\models\varphi$. The cardinal $\kappa$ is said to be $\Pi^1_n$-\emph{indescribable} if $\kappa$ is a $\Pi^1_n$-indescribable subset of $\kappa$. The $\Pi^1_0$-indescribable cardinals are precisely the inaccessible cardinals, and, if $\kappa$ is inaccessible, then $S \subseteq \kappa$ is
$\Pi^1_0$-indescribable if and only if it is stationary. The $\Pi^1_1$-indescribable cardinals are precisely the weakly compact cardinals.

The $\Pi^1_n$-\emph{indescribability ideal} on $\kappa$ is
\[\Pi^1_n(\kappa)=\{X\subseteq\kappa\st \text{$X$ is not $\Pi^1_n$-indescribable}\},\]
the corresponding collection of positive sets is
\[\Pi^1_n(\kappa)^+=\{X\subseteq\kappa\st\text{$X$ is $\Pi^1_n$-indescribable}\}\]
and the dual filter is
\[\Pi^1_n(\kappa)^*=\{\kappa\setminus X\st X\in\Pi^1_n(\kappa)\}.\]
Clearly, if $\kappa$ is not $\Pi^1_n$-indescribable, then $\Pi^1_n(\kappa)=P(\kappa)$.
\Levy\ proved \cite{MR0281606} that if $\kappa$ is $\Pi^1_n$-indescribable, then $\Pi^1_n(\kappa)$ is a nontrivial normal ideal on $\kappa$.

A set $C\subseteq\kappa$ is called \emph{$0$-club} if it is a club. A set $X\subseteq\kappa$ is said to be \emph{$n$-closed} if it contains all of its $\Pi^1_{n-1}$-indescribable reflection points: whenever $\alpha<\kappa$ and $X\cap\alpha$ is $\Pi^1_{n-1}$-indescribable, then $\alpha\in X$ (note that such $\alpha$ must be $\Pi^1_{n-1}$-indescribable). If a set $C\subseteq\kappa$ is both $n$-closed and $\Pi^1_{n-1}$-indescribable, then $C$ is said to be an \emph{$n$-club} subset of $\kappa$. For example, $C\subseteq\kappa$ is $1$-club if and only if it is stationary and contains all of its inaccessible stationary reflection points, and $C\subseteq\kappa$ is $2$-club if and only if it is weakly compact and contains all of its weakly compact reflection points. Building on work of Sun \cite{MR1245524}, Hellsten showed \cite{MR2715639} that when $\kappa$ is a $\Pi^1_n$-indescribable cardinal, a set $S\subseteq\kappa$ is $\Pi^1_n$-indescribable if and only if $S\cap C\neq\emptyset$ for every $n$-club $C\subseteq\kappa$. Thus, when $\kappa$ is $\Pi^1_n$-indescribable, the collection of $n$-club subsets of $\kappa$ generates the filter $\Pi^1_n(\kappa)^*$. In particular, this implies that $n$-club sets are themselves $\Pi^1_n$-indescribable.

For $n<\omega$ and $X\subseteq\kappa$, we define the \emph{$n$-trace of $X$} to be
\[\Tr_n(X)=\{\alpha<\kappa\st X\cap\alpha\in \Pi^1_n(\alpha)^+\}.\]
Notice that when $X=\kappa$, $\Tr_n(\kappa)$ is the set of $\Pi^1_n$-indescribable cardinals below $\kappa$, and in particular $\Tr_0(\kappa)$ is the set of inaccessible cardinals less than $\kappa$.
For uniformity of notation, let us say that an ordinal $\alpha$ is \emph{$\Pi^1_{-1}$-indescribable} if it is a limit ordinal, and if $\alpha$ is a limit ordinal, $S \subseteq \alpha$ is
\emph{$\Pi^1_{-1}$-indescribable} if it is unbounded in $\alpha$. Thus, if $X \subseteq \kappa$, then $\Tr_{-1}(X) = \{\alpha < \kappa \st \sup(X \cap \alpha) = \alpha\}$.

\begin{definition} \label{definition_n_square}
Suppose $n<\omega$ and $\Tr_{n-1}(\kappa)$ is cofinal in $\kappa$. A sequence $\vec{C}=\<C_\alpha\st\alpha\in \Tr_{n-1}(\kappa)\>$ is called a \emph{coherent sequence of $n$-clubs} if
\begin{enumerate}
\item for all $\alpha\in \Tr_{n-1}(\kappa)$, $C_\alpha$ is an $n$-club subset of $\alpha$ and
\item for all $\alpha<\beta$ in $\Tr_{n-1}(\kappa)$, $C_\beta\cap\alpha\in\Pi^1_{n-1}(\alpha)^+$ implies $C_\alpha=C_\beta\cap\alpha$.
\end{enumerate}
We say that a set $C\subseteq\kappa$ is a \emph{thread} through a coherent sequence of $n$-clubs $$\vec{C}=\<C_\alpha\st\alpha\in \Tr_{n-1}(\kappa)\>$$ if $C$ is $n$-club and for all $\alpha\in \Tr_{n-1}(\kappa)$, $C\cap\alpha\in\Pi^1_{n-1}(\alpha)^+$ implies $C_\alpha=C\cap\alpha$.
A coherent sequence of $n$-clubs $\vec{C}=\<C_\alpha\st\alpha\in \Tr_{n-1}(\kappa)\>$ is called a \emph{$\square_n(\kappa)$-sequence} if there is no thread through $\vec{C}$. We say that $\square_n(\kappa)$ holds if there is a $\square_n(\kappa)$-sequence $\vec{C}=\<C_\alpha\st\alpha\in\Tr_{n-1}(\kappa)\>$.
\end{definition}

\begin{remark}
Note that $\square_0(\kappa)$ is simply $\square(\kappa)$. For $n=1$, the principle $\square_1(\kappa)$ states that there is a coherent sequence of $1$-clubs \[\< C_\alpha\st\alpha < \kappa \text{ is inaccessible}\>\] that cannot be threaded.
\end{remark}

\begin{remark} \label{remark_relationship}
Before we prove some basic results about $\square_n(\kappa)$, let us consider a similar principle due to Brickhill and Welch \cite{BrickhillWelch}. We will consider the case $n=1$ in detail. We will refer to the notion of $1$-club defined in \cite{BrickhillWelch} as \emph{strong $1$-club} in order to avoid confusion. A set $C\subseteq \kappa$ is a \emph{strong $1$-club} if it is stationary in $\kappa$ and whenever $C\cap\alpha$ is stationary in $\alpha$ then $\alpha\in C$. This notion of strong $1$-club is precisely the same notion considered in \cite{MR1245524}. Thus, it follows from the results of \cite{MR1245524} that when $\kappa$ is weakly $\Pi^1_1$-indescribable\footnote{Recall that a set $S\subseteq\kappa$ is \emph{weakly $\Pi^1_1$-indescribable} if for all $A\subseteq\kappa$ and all $\Pi^1_1$ sentences $\varphi$, $(\kappa,\in,A)\models\varphi$ implies that there is and $\alpha\in S$ such that $(\alpha,\in,A\cap\alpha)\models\varphi$.} the collection of strong $1$-clubs generates the weak $\Pi^1_1$-indescribable filter. Moreover, when $\kappa$ is weakly compact, the collection of strong $1$-club subsets of $\kappa$ generates the weakly compact filter $\Pi^1_1(\kappa)$. For $S\subseteq\kappa$, Brickhill and Welch \cite{BrickhillWelch} define $\square^1(\kappa)$ to be the principle asserting the existence of a sequence $\<C_\alpha\st\cf(\alpha)>\omega\>$ such that
\begin{enumerate}
\item $C_\alpha$ is a strong $1$-club in $\alpha$,
\item whenever $C_\alpha\cap\beta$ is stationary in $\beta$ we have $C_\beta=C_\alpha\cap\beta$ and
\item there is no $C\subseteq\kappa$ that is strong $1$-club in $\kappa$ such that whenever $C\cap\alpha$ is stationary in $\alpha$ we have $C_\alpha=C\cap\alpha$.
\end{enumerate}
Such a sequence is called a $\square^1(\kappa)$-sequence. Let us show that when $\kappa$ is weakly compact, the Brickhill-Welch principle $\square^1(\kappa)$ implies our principle $\square_1(\kappa)$. Suppose $\<C_\alpha\st\cf(\alpha)>\omega\>$ is a $\square^1(\kappa)$-sequence. Clearly, $\<C_\alpha\st\alpha\in\Tr_0(\kappa)\>$ is a coherent sequence of $1$-clubs. For the sake of contradiction, suppose $C\subseteq\kappa$ is a $1$-club thread through $\<C_\alpha\st\alpha\in\Tr_0(\kappa)\>$. Using the coherence of $\<C_\alpha\st\cf(\alpha)>\omega\>$ it is straightforward to check that $C$ is a strong $1$-club and whenever $C\cap\alpha$ is stationary in $\alpha$ we have $C_\alpha=C\cap\alpha$. This contradicts $\square^1(\kappa)$. Thus $\square^1(\kappa)$ implies $\square_1(\kappa)$. It is not known whether $\square_1(\kappa)$ implies $\square^1(\kappa)$.

Brickhill and Welch also generalized their definition to obtain the principles $\square^n(\kappa)$, and again it is not difficult to see that $\square^n(\kappa)$ implies our principle $\square_n(\kappa)$.
\end{remark}

Generalizing the fact that $\square(\kappa)$ implies $\kappa$ is not weakly compact, let us show that $\square_n(\kappa)$ implies $\kappa$ is not $\Pi^1_{n+1}$-indescribable. To do this, we first recall the Hauser characterization of $\Pi^1_n$-indescribability.

We say that a transitive model $\< M,\in\>$ is a $\kappa$-\emph{model} if $|M| = \kappa$, $\kappa\in M$, $M^{\lt\kappa}\subseteq M$, and $M\models\ZFC^-$ ($\ZFC$ without the power set axiom). It is not difficult to see that if $\kappa$ is inaccessible, then $V_\kappa$ is an element of every $\kappa$-model $M$.

\begin{definition}[Hauser]
Suppose $\kappa$ is inaccessible. For $n\geq 0$, a $\kappa$-model $N$ is \emph{$\Pi^1_n$-correct at $\kappa$} if and only if
\[V_\kappa\models\varphi \iff (V_\kappa\models\varphi)^N\]
for all $\Pi^1_n$-formulas $\varphi$ whose parameters are contained in $N\cap V_{\kappa+1}$.
\end{definition}

\begin{remark}
	Notice that every $\kappa$-model is $\Pi^1_0$-correct at $\kappa$.
\end{remark}

\begin{theorem}[Hauser]\label{theorem_hauser}
The following statements are equivalent for every inaccessible cardinal $\kappa$, every subset $S \subseteq \kappa$, and all $0<n<\omega$.
\end{theorem}

\begin{enumerate}
\item $S$ is $\Pi^1_n$-indescribable.
\item For every $\kappa$-model $M$ with $S\in M$, there is a $\Pi^1_{n-1}$-correct $\kappa$-model $N$ and an elementary embedding $j:M\to N$ with $\crit(j)=\kappa$ such that $\kappa\in j(S)$.
\item For every $A\subseteq\kappa$ there is a $\kappa$-model $M$ with $A,S\in M$ for which there is a $\Pi^1_{n-1}$-correct $\kappa$-model $N$ and an elementary embedding $j:M\to N$ with $\crit(j)=\kappa$ such that $\kappa\in j(S)$.
\item For every $A\subseteq\kappa$ there is a $\kappa$-model $M$ with $A,S\in M$ for which there is a $\Pi^1_{n-1}$-correct $\kappa$-model $N$ and an elementary embedding $j:M\to N$ with $\crit(j)=\kappa$ such that $\kappa\in j(S)$ and $j,M\in N$.
\end{enumerate}

\begin{lemma}\label{lemma_indescribable_union}
Suppose $\kappa$ is a cardinal. If $S\in \Pi^1_n(\kappa)^+$ and $S_\alpha\in\Pi^1_n(\alpha)^+$ for each $\alpha\in S$, then $\bigcup_{\alpha\in S}S_\alpha\in\Pi^1_n(\kappa)^+$.
\end{lemma}

\begin{proof}
Fix an $n$-club $C$ in $\kappa$. The set $\Tr_{n-1}(C)$ is $n$-closed because if $\Tr_{n-1}(C)\cap \alpha\in \Pi^1_{n-1}(\alpha)^+$, then $C\cap\alpha\in \Pi^1_{n-1}(\alpha)^+$ since, by $n$-closure of $C$, $\Tr_{n-1}(C)\subseteq C$. Also, $\Tr_{n-1}(C)$ meets every $n$-club $D$ because the intersection $C\cap D$ is an $n$-club. Thus, $\Tr_{n-1}(C)$ is an $n$-club. It follows that there is an $\alpha\in S\cap \Tr_{n-1}(C)$. Since $S_\alpha$ is $\Pi^1_n$-indescribable in $\alpha$ and $C\cap\alpha$ is an $n$-club in $\alpha$, we have $S_\alpha\cap C\cap\alpha\neq\emptyset$, and hence $\left(\bigcup_{\alpha\in S}S_\alpha\right)\cap C\neq\emptyset$.
\end{proof}

A simple complexity calculation shows that
for every $n<\omega$, there is a $\Pi^1_{n+1}$-formula $\chi_n(X)$ such that for every $\kappa$ and every $S\subseteq\kappa$,
$(V_\kappa,\in)\models\chi_n(S)$ if and only if $S$ is $\Pi^1_n$-indescribable (see \cite[Corollary 6.9]{MR1994835}).  It therefore follows that there is a $\Pi^1_n$-formula $\psi_n(X)$ such that for every $\kappa$ and every $C\subseteq\kappa$, $(V_\kappa,\in)\models\psi_n(C)$ if and only if $C$ is an $n$-club subset of $\kappa$. Thus, in particular, a $\Pi^1_n$-correct model $N$ is going to be correct about $\Pi^1_{n-1}$-indescribable sets as well as $n$-clubs.

\begin{corollary}\label{cor_trace_n-1_indescribable}
Suppose $\kappa$ is $\Pi^1_n$-indescribable. If $S\in \Pi^1_n(\kappa)^+$, then \[\Tr_{n-1}(S)=\{\alpha<\kappa\st S\cap\alpha\in \Pi_{n-1}^1(\alpha)^+\}\] is an $n$-club. 
\end{corollary}

\begin{proof}
Suppose $S$ is $\Pi^1_n$-indescribable. First, let us argue that $\Tr_{n-1}(S)$ is $\Pi^1_n$-indescribable. Let $M$ be a $\kappa$-model with $S,\Tr_{n-1}(S)\in M$ and let $j:M\to N$ be an elementary embedding with critical point $\kappa$ such that $N$ is $\Pi^1_{n-1}$-correct and $\kappa\in j(S)$. The $\Pi^1_{n-1}$-correctness of $N$ implies that $j(S)\cap\kappa=S$ is a $\Pi^1_{n-1}$-indescribable subset of $\kappa$ in $N$. Thus, $\kappa\in j(\Tr_{n-1}(S))$. Hence $\Tr_{n-1}(S)$ is $\Pi^1_n$-indescribable.

It remains to show that $\Tr_{n-1}(S)$ is $n$-closed, which is equivalent to showing that if $\Tr_{n-1}(S)\cap\alpha\in \Pi^1_{n-1}(\alpha)^+$, then $S\cap \alpha\in\Pi^1_{n-1}(\alpha)^+$. More generally, observe that if $X\subseteq\alpha$ and $\Tr_m(X)$ is $\Pi^1_m$-indescribable, then $X=\bigcup_{\beta\in \Tr_m(X)}X\cap \beta$ must be $\Pi^1_m$-indescribable by Lemma~\ref{lemma_indescribable_union}.
\end{proof}

\begin{proposition}\label{proposition_n_square_denies_n_plus_1_indescribability}
For every $n<\omega$, $\square_n(\kappa)$ implies that $\kappa$ is not $\Pi^1_{n+1}$-indescribable.
\end{proposition}

\begin{proof}
Suppose $\vec{C}=\<C_\alpha\st\alpha\in\Tr_n(\kappa)\>$ is a $\square_n(\kappa)$-sequence and $\kappa$ is $\Pi^1_{n+1}$-indescribable. Let $M$ be a $\kappa$-model with $\vec{C}\in M$. Since $\kappa$ is $\Pi^1_{n+1}$-indescribable, we may let $j:M\to N$ be an elementary embedding with critical point $\kappa$ and a $\Pi^1_n$-correct $N$ as in Theorem \ref{theorem_hauser} (2). By elementarity, it follows that $j(\vec{C})=\<\bar{C}_\alpha\st\alpha\in\Tr_n^N(j(\kappa))\>$ is a $\square_n(j(\kappa))$-sequence in $N$. Since $N$ is $\Pi^1_n$-correct, we know that $\kappa \in \Tr_n^N(j(\kappa))$ and $\bar C_\kappa$ must also be $n$-club in $V$. Since $j(\vec{C})$ is a $\square_n(j(\kappa))$-sequence in $N$, it follows that for every $\Pi^1_n$-indescribable $\alpha<\kappa$ if $\bar{C}_\kappa\cap\alpha\in\Pi^1_n(\alpha)^+$, then $\bar{C}_\kappa\cap\alpha=C_\alpha$, and hence $\bar{C}_\kappa$ is a thread through $\vec{C}$, a contradiction.
\end{proof}

Let us now describe the sense in which $\square_n(\kappa)$ can hold trivially when $\kappa$ is $\Pi^1_n$-indescribable and certain reflection principles fail often below $\kappa$.

\begin{definition}\label{definition_n_square_holds_trivially}
Suppose $n<\omega$ and $\Tr_{n-1}(\kappa)$ is cofinal in $\kappa$. We say that \emph{$\square_n(\kappa)$ holds trivially} if there is a $\square_n(\kappa)$-sequence $\vec{C}=\<C_\alpha\st\alpha\in\Tr_{n-1}(\kappa)\>$ and a club $E \subseteq \kappa$ such that for all $\alpha\in\Tr_{n-1}(\kappa) \cap E$, $C_\alpha$ is trivially an $n$-club subset of $\alpha$ in the sense that $C_\alpha$ is a $\Pi^1_{n-1}$-indescribable subset of $\alpha$ and has no $\Pi^1_{n-1}$-indescribable proper initial segment.
\end{definition}

Notice that $\square(\kappa)$ holds trivially if $\cf(\kappa) = \omega_1$. In this case we can find a club $E \subseteq \kappa$
consisting of ordinals of countable cofinality, namely, let $\langle \alpha_\xi:\xi<\omega_1\rangle$ be an increasing continuous cofinal sequence in $\kappa$, and let $E$ consist of $\alpha_\xi$ for $\xi$ a limit ordinal. For all $\alpha \in E$, we can let $C_\alpha$ be a cofinal subset of $\alpha$ of order type $\omega$. Then, for every limit ordinal $\beta \in \kappa \setminus E$, we can let $\alpha_\beta = \max(E \cap \beta)$
and set $C_\beta$ to be the interval $(\alpha_\beta, \beta)$. It is easily verified that a sequence thus defined is a $\square(\kappa)$-sequence.

Recall that the principle $\Refl_n(\kappa)$ holds if and only if $\kappa$ is $\Pi^1_n$-indescribable and
for every $\Pi^1_n$-indescribable subset $X$ of $\kappa$, there is an $\alpha<\kappa$ such that $X\cap \alpha$ is $\Pi^1_n$-indescribable (see \cite{MR3985624} and \cite{MR4050036} for more details).

\begin{proposition}\label{proposition_n_square_can_hold_trivially}
Suppose $1\leq n<\omega$ and $\kappa$ is $\Pi^1_n$-indescribable.
Then $\square_n(\kappa)$ holds trivially if and only if there is a
club $E \subseteq \kappa$ such that  $\lnot\Refl_{n-1}(\alpha)$ holds
for every $\alpha\in\Tr_{n-1}(\kappa) \cap E$.
\end{proposition}

\begin{proof}
If $\square_n(\kappa)$ holds trivially, then there is a $\square_n(\kappa)$-sequence $\vec{C}=\<C_\alpha\st\alpha\in\Tr_{n-1}(\kappa)\>$ and a
club $E \subseteq \kappa$ such that for $\alpha\in\Tr_{n-1}(\kappa) \cap E$, $C_\alpha$ is a $\Pi^1_{n-1}$-indescribable set with no $\Pi^1_{n-1}$-indescribable initial segment, in which case $C_\alpha$ is a witness to the fact that $\Refl_{n-1}(\alpha)$ fails.

Conversely, suppose that $E \subseteq \kappa$ is a club and $\lnot\Refl_{n-1}(\alpha)$ holds for every $\alpha\in\Tr_{n-1}(\kappa) \cap E$.
For each $\alpha\in \Tr_{n-1}(\kappa) \cap E$, let $C_\alpha$ be a $\Pi^1_{n-1}$-indescribable subset of $\alpha$ which has no $\Pi^1_{n-1}$-indescribable proper initial segment. Then each $C_\alpha$ is trivially $n$-club in $\alpha$. For all $\beta \in \Tr_{n-1}(\kappa) \setminus E$, let
$\alpha_\beta = \max(E \cap \beta)$, and let $C_\beta$ be the interval $(\alpha_\beta, \beta)$. Then $\vec{C}=\<C_\alpha\st\alpha\in \Tr_{n-1}(\kappa)\>$ is easily
seen to be a coherent sequence of $n$-clubs, since there are no points at which coherence needs to be checked for indices in $E$ and
coherence is easily checked for indices outside of $E$ because of the uniformity of the definition. We must argue that $\vec{C}$ has no thread. Suppose there is a thread $C\subseteq\kappa$ through $\vec{C}$. Since $\kappa$ is $\Pi^1_n$-indescribable and $C$ is an $n$-club subset of $\kappa$ it follows, by Corollary~\ref{cor_trace_n-1_indescribable}, that $\Tr_{n-1}(C)$ is an $n$-club in $\kappa$. Thus we can choose $\alpha,\beta\in \Tr_{n-1}(C) \cap E$ with $\alpha<\beta$. Since $C$ is a thread we have $C_\alpha=C_\beta\cap\alpha=C\cap\alpha$, which contradicts the fact that $C_\beta$ has no $\Pi^1_{n-1}$-indescribable proper initial segment. This shows that $\square_n(\kappa)$ holds trivially.\end{proof}

\begin{remark}
	It seems like it might be more optimal to change Definition~\ref{definition_n_square_holds_trivially}
	to instead say that $\square_n(\kappa)$ holds trivially if there is
	a $\square_n(\kappa)$-sequence $\vec{C}$ and an $n$-\emph{club} $E \subseteq \kappa$ such that for all
	$\alpha \in \Tr_{n-1}(\kappa) \cap E$, $C_\alpha$ is trivially an $n$-club subset of $\alpha$.
	However, we were not able to prove the analogue of Proposition~\ref{proposition_n_square_can_hold_trivially}
	corresponding to this alternative definition, namely that $\square_n(\kappa)$ holds trivially if and
	only if there is an $n$-\emph{club} $E \subseteq \kappa$ such that $\lnot\Refl_{n-1}(\alpha)$
	holds for every $\alpha \in \Tr_{n-1}(\kappa) \cap E$.
\end{remark}


\begin{corollary}
In $L$, if $\kappa$ is the least $\Pi^1_n$-indescribable cardinal, then $\square_n(\kappa)$ holds trivially.
\end{corollary}

\begin{proof}
Generalizing a result of Jensen \cite{MR0309729}, Bagaria, Magidor and Sakai proved \cite{MR3416912} that in $L$ a cardinal $\kappa$ is $\Pi^1_n$-indescribable if and only if $\Refl_{n-1}(\kappa)$ holds. Suppose $V=L$ and $\kappa$ is the least $\Pi^1_n$-indescribable cardinal. Then $\Refl_{n-1}(\alpha)$ fails for all $\alpha<\kappa$. Hence by Proposition \ref{proposition_n_square_can_hold_trivially}, $\square_n(\kappa)$ holds trivially.
\end{proof}

Brickhill and Welch showed more generally that in $L$, if $\kappa$ is $\Pi^1_n$-indescribable and not $\Pi^1_{n+1}$-indescribable, then their principle $\square^n(\kappa)$ holds. Since $L$ can have $\Pi^1_n$-indescribable, but not $\Pi^1_{n+1}$-indescribable, cardinals below which, for instance, the set of $\Pi^1_{n-1}$-indescribable cardinals is $\Pi^1_n$-indescribable, it follows from reasonable assumptions that in $L$ our principle $\square_n(\kappa)$ can hold nontrivially at a $\Pi^1_n$-indescribable cardinal. We do not know how to force $\square_n(\kappa)$ to hold non-trivially at a $\Pi^1_n$-indescribable cardinal.

Another consequence of Proposition \ref{proposition_n_square_can_hold_trivially} is that we can force $\square_1(\kappa)$ to hold \emph{trivially} at a $\Pi^1_1$-indescribable cardinal by killing certain stationary reflection principles below $\kappa$.

Recall that a partial order $\P$ is said to be $\alpha$-\emph{strategically closed}, for an ordinal $\alpha$, if Player II has a winning strategy in the following two-player game $\mathcal {G}_\alpha(\P)$ of perfect information. In a run of $\mathcal{G}_\alpha(\P)$, the two players take turns playing elements of
a decreasing sequence $\langle p_\beta \st \beta < \alpha \rangle$ of conditions from $\P$. Player I plays at all odd ordinal stages, and Player II plays at all even ordinal stages (in particular, at limits). Player II goes first and must play $\one_{\P}$. Player I wins if there is a limit ordinal $\gamma < \alpha$
such that $\langle p_\beta \st \beta < \gamma \rangle$ has no lower bound (i.e., if Player II is unable to play at stage $\gamma$).
If the game continues successfully for $\alpha$-many moves, then Player II wins. Clearly, for a cardinal $\alpha$, if $\P$ is $\alpha$-strategically closed, then $\P$ is $\lt\alpha$-distributive, and hence adds no new $\alpha$-sequences of ground model sets.

We will use the following general proposition about indestructibility of weakly compact cardinals.
\begin{definition}
Suppose $\kappa$ is an inaccessible cardinal. We say that a forcing iteration $$\<\P_\alpha,\dot{\Q}_\beta\st \alpha\leq\kappa, ~ \beta<\kappa\>$$ is \emph{good} if it has Easton support and, for all
$\alpha < \kappa$, if $\alpha$ is inaccessible, then $\dot{\Q}_\alpha$ is a $\P_\alpha$-name for a poset such that $\one_{\P_\alpha}\forces \dot \Q_\alpha\in \dot V_\kappa$,
where $\dot{V}_\kappa$ is a $\P_\alpha$-name for $(V_\kappa)^{V^{\P_\alpha}}$ and, otherwise, $\dot{\Q}_\alpha$ is a $\P_\alpha$-name for trivial forcing.

\end{definition}
If $\P_\kappa$ is a good iteration, then we can argue by induction on $\alpha$ that every $\P_\alpha\in V_\kappa$ because if $\P_\alpha\in V_\kappa$ and $\one_{\P_\alpha}\forces \dot \Q_\alpha\in \dot V_\kappa$, then $\P_\alpha*\dot{\Q}_\alpha\in V_\kappa$. The following standard proposition about good iterations can be found, for example, in \cite{MR2768691}.
\begin{proposition}\label{prop_goodIterations}
Suppose $\kappa$ is a Mahlo cardinal. Then a good iteration $\P_\kappa$ has size $\kappa$ and is $\kappa$-c.c.
\end{proposition}
\begin{lemma}\label{lemma_indestructibleWC}
Suppose $\kappa$ is weakly compact and $\<\P_\alpha,\dot{\Q}_\beta\st \alpha\leq\kappa, ~ \beta<\kappa\>$ is a good iteration which at non-trivial stages $\alpha$ has $\one_{\P_\alpha}\forces ``\dot \Q_\alpha\text{ is }\alpha\text{-strategically closed}$", and let $G$ be $\P$-generic over $V$. Then $\kappa$ remains weakly compact in $V[G]$.
\end{lemma}
\begin{proof}
By Proposition~\ref{prop_goodIterations}, we can assume without loss that $\P\subseteq V_\kappa$. Since $\kappa$ is weakly compact, there are $\kappa$-models $M$ and $N$ with $A\in M$ for which there is an elementary embedding $j:M\to N$ with critical point $\kappa$. A nice-name counting argument, using the $\kappa$-c.c. and the fact that the tails of the forcing iteration are eventually $\alpha$-distributive for every $\alpha<\kappa$, shows that $\kappa$ is inaccessible in $V[G]$.

Suppose $A\in P(\kappa)^{V[G]}$ and let $\dot{A}\in H(\kappa^+)^V$ be a $\P_\kappa$-name such that $\dot{A}_{G}=A$. Let $M$ be a $\kappa$-model with $\dot{A}\in M$ for which there are a
$\kappa$-model $N$ and an elementary embedding $j:M\to N$ with critical point $\kappa$. Since $N^{<\kappa}\cap V\subseteq N$, we have $j(\P_\kappa)\cong \P_\kappa*\dot\Q_\kappa*\dot{\P}_{\kappa,j(\kappa)}$, where $N$ believes that $\one_{\P_\kappa}\forces ``\dot \Q_\kappa$ is $\kappa$-strategically closed", and $\dot{\P}_{\kappa,j(\kappa)}$ is a $\P_\kappa*\dot\Q_\kappa$-name for $N$'s version of the tail of the iteration $j(\P_\kappa)$ of length $j(\kappa)$. By the generic closure criterion (Lemma~\ref{lemma_genericClosureCriterion}), since $\P_\kappa$ has the $\kappa$-c.c., $N[G]$ is a $\kappa$-model in $V[G]$. The poset $(\dot{\Q}_\kappa*\dot{\P}_{\kappa,j(\kappa)})_{G}$ is $\kappa$-strategically closed in $N[G]$, so, by diagonalizing, we can build an $N[G]$-generic filter $H*G'\in V[G]$ for $(\dot{\Q}_\kappa*\dot{\P}_{\kappa,j(\kappa)})_{G}$. Since conditions in $\P_\kappa$ have supports of size less than the critical point of $j$ we have $j\image G\subseteq \hat{G}=_{\defn}G*H*G'$. Thus $j$ lifts to $j:M[G]\to N[\hat{G}]$. Since $A=\dot{A}_{G}\in M[G]$, this shows that $\kappa$ remains weakly compact in $V[G]$.
\end{proof}

\begin{proposition}\label{theorem_n_square_can_be_trivially_forced}
If $\kappa$ is $\Pi^1_1$-indescribable (weakly compact), then there is a forcing extension in which $\square_1(\kappa)$ holds trivially and $\kappa$ remains $\Pi^1_1$-indescribable.
\end{proposition}

\begin{proof}
For regular $\alpha>\omega$, let $\S_\alpha$ denote the usual forcing to add a nonreflecting stationary subset of $\alpha\cap\cof(\omega)$ (see Example~6.5 in \cite{MR2768691}). Recall that conditions in $\S_\alpha$ are bounded subsets $p$ of $\alpha\cap\cof(\omega)$ such that for every $\beta\leq\sup(p)$ with $\cf(\beta)>\omega$, the set $p\cap \beta$ is nonstationary in $\beta$.
It is not difficult to see that the poset $\S_\alpha$ is $\alpha$-strategically closed.

Now we let $\<\P_\alpha,\dot{\Q}_\beta\st \alpha\leq\kappa,\beta<\kappa\>$ be an Easton-support iteration of length $\kappa$ such that if $\alpha<\kappa$ is inaccessible, then $\dot{\Q}_\alpha$ is a $\P_\alpha$-name for $\S_\alpha^{V^{\P_\alpha}}$, and
otherwise $\dot{\Q}_\alpha$ is a $\P_\alpha$-name for trivial forcing.

Suppose $G$ is generic for $\P_\kappa$ over $V$. By Lemma~\ref{lemma_indestructibleWC}, since $\P_\kappa$ has all the right properties, $\kappa$ remains weakly compact in $V[G]$. Also, in $V[G]$, for each inaccessible $\alpha<\kappa$, by a routine genericity argument and the fact that the tail of the forcing iteration from stage $\alpha + 1$ to $\kappa$ is $\alpha^+$-strategically closed, the stage $\alpha$ generic $H_\alpha$ obtained from $G$ yields a nonreflecting stationary subset of $\alpha$: $S_\alpha=\bigcup H_\alpha$. Thus in $V[G]$, $\Refl_0(\alpha)$ fails for all inaccessible $\alpha<\kappa$, and hence $\square_1(\kappa)$ holds trivially by Proposition \ref{proposition_n_square_can_hold_trivially}.
\end{proof}
In Section~\ref{section_1_sqare_at_a_weakly_compact} we will show that $\square_1(\kappa)$ can hold non-trivially at a weakly compact cardinal.

\section{Preserving $\Pi^1_n$-indescribability by forcing}
In this section, we will provide some results to be used in indestructibility arguments for $\Pi^1_n$-indescribable cardinals in later sections.

The following two folklore lemmas (and their variants) are widely used in indestructibility arguments for large cardinals characterized by the existence of elementary embeddings.
\begin{lemma}[Ground closure criterion]\label{lemma_groundClosureCriterion}
Suppose $\kappa$ is a cardinal, $M$ is a $\kappa$-model, $\P\in M$ is a forcing notion, and $G\in V$ is generic for $\P$ over $M$. Then $M[G]$ is a $\kappa$-model.
\end{lemma}
\begin{lemma}[Generic closure criterion]\label{lemma_genericClosureCriterion}
Suppose $\kappa$ is a cardinal, $M$ is a $\kappa$-model, $\P \in M$ is a forcing notion with the $\kappa$-c.c., and $G$ is generic for $\P$ over $V$. Then $M[G]$ is a $\kappa$-model in $V[G]$.
\end{lemma}

\begin{lemma}\label{lemma_strategic}
Suppose $\kappa$ is inaccessible, $\P$ is a $\kappa$-strategically closed forcing and $G$ is generic for $\P$ over $V$. Then $(V_\kappa,\in,A)\models\forall X\psi(X,A)$ implies $((V_\kappa,\in,A)\models\forall X\psi(X,A))^{V[G]}$ for all $A\in V_{\kappa+1}^V$ and all first order $\psi$.
\end{lemma}

\begin{proof}
First, observe that since $\P$ is $\lt\kappa$-distributive, $\kappa$ remains inaccessible in $V[G]$ and $V_\kappa=V_\kappa^{V[G]}$.
Suppose towards a contradiction that $(V_\kappa,\in,A)\models\forall X\psi(X,A)$, but for some $B\subseteq V_\kappa$ in $V[G]$, $(V_\kappa,\in,A)\models\neg\psi(B,A)$. Let $\dot B$ be a $\P$-name for $B$. Since $\kappa$ is inaccessible in $V[G]$, the set $$C=\{\alpha<\kappa\st (V_\alpha,\in,A\cap\alpha,B\cap\alpha)\models\lnot\psi(B\cap\alpha,A\cap\alpha))\}$$ contains a club in $V[G]$. Let $\dot{C}$ be a $\P$-name for such a club. In $V$, we can use Player II's winning strategy in $\mathcal{G}_\kappa(\P)$ together with the names $\dot{B}$ and $\dot{C}$ to build $\hat{B}$ and $\hat{C}$ such that $\hat{C}$ is club in $\kappa$ and for each $\alpha\in\hat{C}$ we have $$(V_\alpha,\in,A\cap V_\alpha,\hat{B}\cap V_\alpha)\models\lnot\psi(\hat{B}\cap V_\alpha,A\cap V_\alpha).$$ Since $(V_\kappa,\in,A)\models\forall X\psi(X,A)$, we have $(V_\kappa,\in,A,\hat{B})\models \psi(\hat{B},A)$, and since $\kappa$ is inaccessible, the set $$\{\alpha<\kappa\st (V_\alpha,\in,A\cap V_\alpha,\hat{B}\cap V_\alpha)\models\psi(\hat{B}\cap\alpha,A\cap\alpha)\}$$ contains a club. Thus, there is an $\alpha\in \hat{C}$ such that $$(V_\alpha,\in,A\cap V_\alpha,\hat{B}\cap V_\alpha)\models\psi(\hat{B}\cap V_\alpha,A\cap V_\alpha),$$ a contradiction.
\end{proof}
\begin{corollary}\label{cor_kappaStatClosedPreservesPi11}
Suppose $\kappa$ is inaccessible, $\P$ is a $\kappa$-strategically closed forcing notion and $G$ is generic for $\P$ over $V$. If $N$ is a $\Pi^1_1$-correct $\kappa$-model in $V$, then $N$ remains a $\Pi^1_1$-correct $\kappa$-model in $V[G]$.
\end{corollary}
\begin{proof}
Clearly $N$ remains a $\kappa$-model because $\P$ is $\lt\kappa$-distributive.
Let $\varphi$ be a $\Pi^1_1$-statement, and suppose first that $(V_\kappa\models\varphi)^N$.
By $\Pi^1_1$-correctness, $V_\kappa\models\varphi$, and so by Lemma~\ref{lemma_strategic},
$(V_\kappa\models\varphi)^{V[G]}$. On the other hand,
if $(V_\kappa \models \lnot \varphi)^N$, then there is a $B \subseteq V_\kappa$ in $N$ witnessing
this failure. Since $N$, $V$, and $V[G]$ all have the same $V_\kappa$, $B$ witnesses the failure
of $\varphi$ in both $V$ and $V[G]$ as well, so $(V_\kappa \models \lnot \varphi)^{V[G]}$
\end{proof}
\begin{proposition}\label{proposition_nonresurection}
Suppose $\kappa$ is inaccessible, $\P$ is $\kappa$-strategically closed, and $G$ is generic for $\P$ over $V$. If $S\in P(\kappa)^V$ is $\Pi^1_1$-indescribable in $V[G]$, then $S$ is $\Pi^1_1$-indescribable in $V$.
\end{proposition}

\begin{proof}
Suppose towards a contradiction that there is $S\in P(\kappa)^V$ that is $\Pi^1_1$-indescribable in $V[G]$ but not $\Pi^1_1$-indescribable in $V$. In $V$, find a subset $A\subseteq V_\kappa$ and a $\Pi^1_1$ statement $\varphi=\forall X\psi(X,A)$ such that $(V_\kappa,\in,A)\models\varphi$ and for all $\alpha\in S$ we have $(V_\alpha,\in,A\cap V_\alpha)\models\lnot\varphi$. Since $\P$ is $\lt\kappa$-distributive, $V$ and $V[G]$ have the same $V_\kappa$, so it follows that in $V[G]$, by the $\Pi^1_1$-indescribability of $S$, it must be the case that $(V_\kappa,\in,A)\models\exists X\lnot\psi(X,A)$. Working in $V[G]$, we fix $B\subseteq V_\kappa$ such that \hbox{$(V_\kappa,\in,A)\models\neg\psi(B,A)$} and observe that the set $$C=\{\alpha<\kappa\st V_\alpha\models\lnot\psi(B\cap V_\alpha,A\cap V_\alpha)\}$$ contains a club. Let $\dot{C}$ be a $\P$-name for such a club, and let $\dot{B}$ be a $\P$-name for $B$. In $V$, we can use
Player II's winning strategy in $\mathcal{G}_\kappa(\P)$ together with $\dot{B}$ and $\dot{C}$ to build $\hat{B}$ and $\hat{C}$ such that $\hat{C}\subseteq\kappa$ is club and $\forall\alpha\in\hat{C}$, $V_\alpha\models\lnot\psi(\hat{B}\cap V_\alpha,A\cap V_\alpha)$. But this implies that $V_\kappa\models\lnot\psi(\hat{B},A)$, a contradiction.
\end{proof}
The converse of Proposition~\ref{proposition_nonresurection} is clearly false because the forcing ${\rm Add}(\kappa,1)$ to add a Cohen subset to $\kappa$ with bounded conditions can destroy the weak compactness of $\kappa$ and it is $\lt\kappa$-closed and therefore $\kappa$-strategically closed.
We will see in Section~\ref{sec_squareAndReflection} (Remark~\ref{remark_nonresurectionfails}) that Proposition~\ref{proposition_nonresurection} can fail for $\Pi^1_2$-indescribable sets.

A good iteration $\P_\kappa$ of length $\kappa$ is said to be \emph{progressively closed} if for every $\alpha<\kappa$, there is $\alpha \leq \beta_\alpha<\kappa$ such that every stage after $\beta_\alpha$ is forced to be $\alpha$-strategically closed. In this case, it is not difficult to see that $\P_{\beta_\alpha}$ forces that the tail of the iteration is $\alpha$-strategically closed.
Next, we will show that good progressively closed $\kappa$-length iterations preserve $\Pi^1_n$-correctness.



Let $\P$ be a forcing notion and suppose $\sigma$ is a $\P$-name. Recall that $\tau$ is a \emph{nice name for a subset of $\sigma$} if $$\tau=\bigcup(\{\pi\}\times A_\pi\st \pi\in\dom(\sigma)\},$$ where each $A_\pi$ is an antichain of $\P$. It is well known and easy to verify that for every $\P$-name $\mu$, there is a nice name $\tau$ for a subset of $\sigma$ such that $\one_{\P}\forces \mu\subseteq\sigma\rightarrow\mu=\tau$. We call such $\tau$ the \emph{nice replacement for $\mu$}.

\begin{lemma}\label{lemma_nice_names}
Suppose $\sigma$ is a $\P$-name and $n\geq 0$. Let $X_\sigma$ be the set of nice names for subsets of $\sigma$, let $p$ be a condition in $\P$ and let $\varphi$ be any $\Pi_n$-assertion in the forcing language of the form
\[(\forall x_1\subseteq\sigma)(\exists x_2\subseteq\sigma)\cdots\psi(x_1,\ldots,x_n).\]
Then $p\forces\varphi$ if and only if
\[(\forall \tau_1\in X_\sigma)(\exists \tau_2\in X_\sigma)\cdots p\forces\psi(\tau_1,\ldots,\tau_n).\]
The analogous statement holds for $\Sigma_n$-assertions in the forcing language.
\end{lemma}

\begin{proof}
We will prove the lemma simultaneously for $\Pi_n$ and $\Sigma_n$ statements by induction on $n$. Clearly the lemma holds for $n=0$. Assume inductively that the lemma holds for some $n$, and suppose $\varphi$ is an assertion in the forcing language of complexity $\Pi_{n+1}$. Let $$\varphi=(\forall x_1\subseteq\sigma)(\exists x_2\subseteq\sigma)\cdots\psi(x_1,\ldots,x_{n+1})=(\forall x_1\subseteq\sigma)\bar{\varphi}(x_1),$$ where $\psi$ is $\Delta_0$ and $\bar{\varphi}(x)$ is $\Sigma_n$. For the forward direction, clearly $p\forces \varphi$ implies $(\forall \tau_1\in X_\sigma)(p\forces \bar{\varphi}(\tau_1))$. By the inductive hypothesis applied to $p\forces\bar{\varphi}(\tau_1)$, we conclude that $(\forall \tau_1\in X_\sigma)(\exists \tau_2\in X_\sigma)\cdots p\forces\psi(\tau_1,\ldots,\tau_n)$. For the converse, suppose $(\forall\tau_1\in X_\sigma)(\exists\tau_2\in X_\sigma)\cdots p\forces\psi(\tau_1,\ldots,\tau_n)$ holds. Let us argue that $p\forces(\forall x_1\subseteq\sigma)(\exists x_2\subseteq\sigma)\cdots\psi(x_1,\ldots,x_n)$. If not, there is some $q\leq p$ and some $\P$-name $\mu$ for a subset of $\sigma$ such that $q\forces (\forall x_2\subseteq\sigma)\cdots \lnot\psi(\mu,x_2,\ldots,x_n)$. Let $\tau$ be a nice replacement for $\mu$ so that $q\forces (\forall x_2\subseteq\sigma)\cdots\lnot\psi(\tau,x_2,\ldots,x_n)$, or in other words, $q\forces\lnot\bar{\varphi}(\tau)$. By assumption $(\exists\tau_2\in X_\sigma)\cdots p\forces\psi(\tau,\tau_2,\ldots,\tau_n)$, so applying the inductive hypothesis, we obtain $p\forces (\exists x_2\subseteq\sigma)\cdots\psi(\tau,x_2,\ldots, x_n)$ and hence $p\forces\bar{\varphi}(\tau)$, a contradiction. The proof of the lemma for $\Sigma_{n+1}$ statements is similar.
\end{proof}

\begin{theorem}\label{theorem_easton_iterations_preserve_correctness}
Suppose $\kappa$ is a Mahlo cardinal, $N$ is a $\Pi^1_n$-correct $\kappa$-model and $\P\in N$ is a progressively closed good Easton-support iteration of length $\kappa$. If $G\subseteq \P$ is generic over $V$, then $N[G]$ is a $\Pi^1_n$-correct $\kappa$-model in $V[G]$.
\end{theorem}

\begin{proof}
By Proposition~\ref{prop_goodIterations}, $\P$ has the $\kappa$-c.c. and without loss of generality $\P\subseteq V_\kappa$. Thus, by the generic closure criterion Lemma~\ref{lemma_genericClosureCriterion}, $N[G]$ remains a $\kappa$-model in $V[G]$. By the progressive closure of the iteration, $V_\kappa^{V[G]}=V_\kappa[G]$. Thus, $V_\kappa^{N[G]}=V_\kappa^{V[G]}$. Let $\sigma\in N$ be a $\P$-name such that $\sigma_G=V_\kappa^{N[G]}=V_\kappa^{V[G]}$ and $\dom(\sigma)\subseteq V_\kappa$.

Let us argue that $N[G]$ is $\Pi^1_n$-correct. Suppose $(V_\kappa^{N[G]},\in,A)\models\varphi$ in $N[G]$, where
\[\varphi = \forall X_1\exists X_2\cdots \psi(X_1,\ldots,X_n,A)\]
is $\Pi^1_n$ and all quantifiers appearing in $\psi$ are first-order over $V_\kappa^{N[G]}$. Let $\dot{A}$ be a $\P$-name for $A$ such that $\dom(\dot{A})\subseteq V_\kappa$. Let $\bar{\psi}(x_1,\ldots,x_n,\dot{A})$ be a formula in the forcing language obtained from $\psi$ by replacing all parameters with $\P$-names and all first-order quantifiers ``$Qx$'' with ``$Qx\in\sigma$'' for $Q=\forall,\exists$. Let $\bar{\varphi}(\sigma,\dot{A})$ denote the following formula in the forcing language:
\[(\forall x_1\subseteq\sigma)(\exists x_2\subseteq\sigma)\cdots \bar{\psi}(x_1,\ldots,x_n,\dot{A}).\]
Since $(V_\kappa^{N[G]},\in,A)\models\varphi$ holds in $N[G]$, it follows that $N[G]\models\bar{\varphi}(\sigma_G,\dot{A}_G)$. Thus, we may choose $p\in G$ with $(p\forces\bar{\varphi}(\sigma,\dot{A}))^N$. By Lemma \ref{lemma_nice_names},
\begin{align}
(\forall\tau_1\in X_\sigma)(\exists\tau_2\in X_\sigma)\cdots p\forces\psi(\tau_1,\ldots,\tau_n,\dot{A})\label{equation_expressing_phi}
\end{align}
holds in $N$. The statement $p\forces\psi(\tau_1,\ldots,\tau_n,\dot{A})$ is first-order in the structure $(V_\kappa,\in,\tau_1,\ldots,\tau_n,\sigma,\dot{A},\P)$.\footnote{This can be proved by using the definition of the forcing relation and induction on complexity of formulas.} Furthermore, since ``$\tau\in X_\sigma$'' can be expressed by a first-order formula $\chi(\tau,\sigma)$ over $(V_\kappa,\in,\sigma,\tau,\P)$, it follows that the statement in (\ref{equation_expressing_phi}) is $\Pi^1_n$ over $(V_\kappa,\in,\sigma,\dot{A})$. Since $N\models$ ``(\ref{equation_expressing_phi}) holds in $(V_\kappa,\in,\sigma,\dot{A})$'' and $N$ is $\Pi^1_n$-correct at $\kappa$, it follows that (\ref{equation_expressing_phi}) holds in $(V_\kappa,\in,\sigma,\dot{A})$. Hence by Lemma \ref{lemma_nice_names}, $p\forces\bar{\varphi}(\sigma,\dot{A})$ over $V$, and since $p\in G$, we conclude that $V[G]\models\bar{\varphi}(\sigma,\dot{A}_G)$, which implies $(V_\kappa^{V[G]},\in,A)\models\varphi$ in $V[G]$.

An analogous argument establishes the converse, verifying that if\break $(V_\kappa^{V[G]},\in,A)\models\varphi$ for a $\Pi^1_n$-assertion $\varphi$ and $A\in N[G]$, then the same assertion holds in $N[G]$.
\end{proof}

A similar argument yields the following result.

\begin{corollary}\label{corollary_closed_forcing_preserves_correctness}
Suppose $\kappa$ is an inaccessible cardinal, $N$ is a $\Pi^1_n$-correct $\kappa$-model and $\P\in N$ is a ${<}\kappa$-distributive forcing notion of size $\kappa$. If $G\subseteq\P$ is generic over $V$, then $N[G]$ remains a $\Pi^1_n$-correct $\kappa$-model in $V[G]$.
\end{corollary}

\begin{proof}
Without loss of generality we can assume that $\P\subseteq V_\kappa$. Since $\P$ is ${<}\kappa$-distributive, $N$ remains a $\kappa$-model in $V[G]$, and, since $G\in V[G]$, it follows that $N[G]$ is a $\kappa$-model in $V[G]$ by the ground closure criterion Lemma~\ref{lemma_groundClosureCriterion}. The ${<}\kappa$-distributivity of $\P$ entails that $V_\kappa^{N[G]}=V_\kappa^{V[G]}=V_\kappa$. Since the statement ``$\tau$ is a nice name for a subset of $\check{V}_\kappa$'' is first-order over the structure $(V_\kappa,\in,\tau,\P)$, the rest of the argument can be carried out as in the proof of Theorem \ref{theorem_easton_iterations_preserve_correctness}.
\end{proof}

The conclusion of Corollary~\ref{corollary_closed_forcing_preserves_correctness} need not hold if the $N$-generic filter $G$ is not fully $V$-generic (see Remark~\ref{rem_groundmodelGenericBreaksCorrectness}).

\section{Shooting $n$-clubs}

Hellsten \cite{MR2653962} showed that if $W\subseteq\kappa$ is any $\Pi^1_1$-indescribable (i.e., weakly compact) subset of $\kappa$, then there is a forcing extension in which $W$ contains a $1$-club and all weakly compact subsets of $W$ remain weakly compact.
We will define a generalization of Hellsten's forcing to shoot an $n$-club through a $\Pi^1_n$-indescribable subset of a cardinal $\kappa$ while preserving the $\Pi^1_n$-indescribability of all its subsets, so that, in particular, $\kappa$ remains $\Pi^1_n$-indescribable in the forcing extension.

Suppose $\gamma$ is an inaccessible cardinal and $A\subseteq\gamma$ is cofinal. For $n\geq 1$, we define a poset $T^n(A)$ consisting of all bounded $n$-closed $c\subseteq A$ ordered by end extension: $c\leq d$ if and only if $d=c\cap\sup_{\alpha\in d}(\alpha+1)$.

\begin{lemma}
For $n\geq 1$, if $\gamma$ is inaccessible and $A\subseteq\gamma$ is cofinal, then $T^n(A)$ is $\gamma$-strategically closed.
\end{lemma}

\begin{proof}
We describe a winning strategy for player II in the game $\mathcal{G}_\kappa(T^n(A))$. Player II begins the game by playing $c_0=\emptyset$. At an even successor stage $\alpha+2$, player II chooses a condition $c_{\alpha+2}\in T^n(A)$ such that $c_{\alpha+2}\lneq c_{\alpha+1}$. At limit stages $\alpha<\gamma$, player II records an ordinal $\gamma_\alpha=\bigcup_{\beta<\alpha}c_\beta$, chooses an element $\eta_\alpha\in A\setminus(\gamma_\alpha+1)$ and plays $c_\alpha=\left(\bigcup_{\beta<\alpha}c_\beta\right) \cup\{\eta_\alpha\}$. In order to argue that $c_\alpha$ is a condition in $T^n(A)$, we need to verify, letting $c=\bigcup_{\beta<\alpha}c_\beta$, that $c$ is not a $\Pi^1_{n-1}$-indescribable subset of $\gamma_\alpha$.  We can assume that
$\gamma_\alpha$ is $\Pi^1_{n-1}$-indescribable, as otherwise $c \cap \gamma_\alpha$ is clearly not $\Pi^1_{n-1}$-indescribable. But then, by construction,
$\{\gamma_\xi \st \xi < \alpha \text{ is a limit ordinal}\}$ is a club (and hence an $(n-1)$-club) in $\gamma_\alpha$ disjoint from $c$, which implies that
$c$ is not a $\Pi^1_{n-1}$-indescribable subset of $\gamma_\alpha$. Thus, $c_\alpha$ is a valid play by Player II, and we have described a winning strategy
in $\mathcal{G}_\kappa(T^n(A))$.
\end{proof}

\begin{remark} \label{generalized_closure_remark}
	In fact, for $n \geq 2$, $T^n(A)$ satisfies the following strengthening of $\gamma$-strategic closure.
	For $X \subseteq \gamma$ and a poset $\P$, let $\mathcal{G}_{\gamma, X}(\P)$ be the modification of
	$\mathcal{G}_{\gamma}(\P)$ in which Player I plays at all stages indexed by an ordinal in $X$ and
	Player II plays elsewhere, and it is still the case that Player I wins if and only if there is a limit ordinal
	$\beta < \gamma$ such that $\langle p_\alpha \st \alpha < \beta \rangle$ has no lower bound in $\P$.
	So, $\mathcal{G}_\gamma(\P)$ is precisely the game $\mathcal{G}_{\gamma, X}(\P)$, where $X$ is the set of
	odd ordinals less than $\gamma$. A routine modification of the proof of the preceding lemma shows that Player II
	has a winning strategy in the game $\mathcal{G}_{\gamma, X}(T^n(A))$ if, for all $\beta < \gamma$,
	$X \cap \beta$ is not $\Pi^1_{n-1}$-indescribable. In particular, this is the case if $X$ is the set
	of all $\alpha < \gamma$ such that $\alpha$ is not $\Pi^1_{n-2}$-indescribable.
\end{remark}

\begin{theorem} \label{theorem_n_club_shooting}
Suppose that $n \geq 1$ and $S\subseteq\kappa$ is $\Pi^1_n$-indescribable. Then there is a forcing extension in which $S$
contains a $1$-club and all $\Pi^1_n$-indescribable subsets of $S$ from $V$ remain $\Pi^1_n$-indescribable.
\end{theorem}

\begin{proof}

Let $\P_{\kappa+1}=\<(\P_\alpha,\dot{\Q}_\beta)\st\alpha\leq\kappa+1, ~ \beta\leq\kappa\>$ be an Easton-support iteration such that
\begin{itemize}
\item if $\gamma\leq\kappa$ is inaccessible and $S\cap\gamma$ is cofinal in $\gamma$, then $\dot{\Q}_\gamma=(T^1(S\cap\gamma))^{V^{\P_\gamma}}$;
\item otherwise, $\dot{\Q}_\gamma$ is a $\P_\gamma$-name for trivial forcing.
\end{itemize}
\noindent Since $\kappa$ is $\Pi^1_n$-indescribable, Proposition~\ref{prop_goodIterations} implies that
$\P_\kappa$ has size $\kappa$ and the $\kappa$-c.c.. Forcing with $\P_{\kappa+1}$ therefore preserves the
inaccessibility of $\kappa$ because $\P_\kappa$ has the $\kappa$-c.c. and is progressively closed and
$\dot\Q_\kappa$ is forced to be $\lt\kappa$-distributive.

Suppose $G*H\subseteq\P_\kappa*\dot{\Q}_\kappa$ is generic over $V$. Clearly, $C(\kappa)=_{\defn}\bigcup H$ is a $1$-closed subset of $S$; to show that $C(\kappa)$ is a $1$-club subset of $\kappa$, it remains to show that $C(\kappa)$ is a stationary subset of $\kappa$ in $V[G*H]$.

Suppose $T\subseteq S$ is $\Pi^1_n$-indescribable in $V$. We will simultaneously show that in $V[G*H]$, $C(\kappa)$ intersects every club subset of $\kappa$ and $T$ remains $\Pi^1_n$-indescribable (in particular, $\kappa$ remains $\Pi^1_n$-indescribable). Fix $A,C\in P(\kappa)^{V[G*H]}$ such that $C$ is a club subset of $\kappa$ in $V[G*H]$. Let $\dot{A}, \dot{C},\dot{C}(\kappa)\in H(\kappa^+)$ be $\P_{\kappa+1}$-names such $\dot{A}_{G*H}=A$, $\dot{C}_{G*H}=C$ and $\dot{C}(\kappa)_{G*H}=C(\kappa)$. In $V$, let $M$ be a $\kappa$-model with $\dot{A},\dot{C},\dot{C}(\kappa),\P_{\kappa+1},T,S\in M$. Since $T$ is $\Pi^1_n$-indescribable in $V$, it follows by Theorem \ref{theorem_hauser} that there is a $\Pi^1_{n-1}$-correct $\kappa$-model $N$ and an elementary embedding $j:M\to N$ with critical point $\kappa$ such that $\kappa\in j(T)$.

Since $N^{<\kappa}\cap V\subseteq N$ and $j(S)\cap\kappa=S$, it follows that $j(\P_\kappa)\cong\P_\kappa* \dot{T}^1(S)*\dot{\P}_{\kappa,j(\kappa)}$, where $\dot{\P}_{\kappa,j(\kappa)}$ is a $\P_{\kappa+1}$-name for the tail of the iteration $j(\P_\kappa)$. Since $\P_\kappa$ has the $\kappa$-c.c., by the generic closure criterion (Lemma~\ref{lemma_genericClosureCriterion}), $N[G]$ is a $\kappa$-model in $V[G]$. Since $T^1(S)$ is $\kappa$-strategically closed, $N[G]$ remains a $\kappa$-model in $V[G*H]$, and hence by the ground closure criterion (Lemma~\ref{lemma_groundClosureCriterion}), $N[G*H]$ is a $\kappa$-model in $V[G*H]$. Since $\P_{\kappa,j(\kappa)}=(\dot{\P}_{\kappa,j(\kappa)})_{G*H}$ is $\kappa$-strategically closed in $N[G*H]$ and $N[G*H]$ is a $\kappa$-model in $V[G*H]$, it follows that there is a filter $G'\in V[G*H]$ which is generic for $\P_{\kappa,j(\kappa)}$ over $N[G*H]$ and the embedding $j$ lifts to $j:M[G]\to N[\hat{G}]$, where $\hat{G}\cong G*H*G'$.

Notice that $p=C(\kappa)\cup\{\kappa\}=\bigcup H\cup\{\kappa\}\in N[\hat{G}]$. Since $\kappa\in j(T)\subseteq j(S)$, we see that $N[\hat{G}]\models$ ``$p$ is a closed subset of $j(S)$''. Thus, $p\in j(T^1(S))$. Since $j(T^1(S))$ is $j(\kappa)$-strategically closed in $N[\hat{G}]$ and $N[\hat{G}]$ is a $\kappa$-model in $V[G*H]$ by the ground closure criterion, there is a filter $\hat{H}\in V[G*H]$ generic for $j(T^1(S))$ over $N[\hat{G}]$ with $p\in \hat{H}$. Since $p$ is below every condition in $j\image H$, we have $j\image H\subseteq \hat{H}$, and thus $j$ lifts to $j:M[G*H]\to N[\hat{G}*\hat{H}]$, where $\kappa\in j(C(\kappa))$. By Theorem \ref{theorem_easton_iterations_preserve_correctness} and Corollary \ref{corollary_closed_forcing_preserves_correctness},
$N[G * H]$ is a $\Pi^1_{n-1}$-correct $\kappa$-model in $V[G*H]$. Since $\bb{P}_{\kappa, j(\kappa)}$ and $j(T^1(S))$ are $(\kappa+1)$-strategically closed in $N[G*H]$,
it follows that $N[G*H]$ and $N[\hat{G} * \hat{H}]$ have the same subsets of $V_\kappa$, so, in particular,
$N[\hat{G}*\hat{H}]$ is a $\Pi^1_{n-1}$-correct $\kappa$-model in $V[G*H]$. Thus, by Theorem \ref{theorem_hauser}, we have verified that $T$ remains $\Pi^1_n$-indescribable in $V[G*H]$.

It remains to show that $C(\kappa)\cap C\neq\emptyset$. Recall that $C$ is a club subset of $\kappa$ in $V[G*H]$,
so $j(C)$ is a club subset of $j(\kappa)$ in $N[\hat{G}*\hat{H}]$. Since $j(C) \cap \kappa = C$, it follows
that $\kappa \in j(C)$, and hence $\kappa \in j(C(\kappa) \cap C)$. By elementarity, $C(\kappa) \cap C \neq
\emptyset$, so $C(\kappa)$ is a stationary and hence 1-club subset of $\kappa$ in $V[G*H]$.
\end{proof}

\begin{remark}
	In the proof of Theorem~\ref{theorem_n_club_shooting}, for any $m \leq n$ we can force with $T^m(S \cap \gamma)$ at every relevant
	$\gamma \leq \kappa$ instead of $T^1(S \cap \gamma)$. This iteration will still preserve
	the $\Pi^1_n$-indescribability of every subset of $S$ that is $\Pi^1_n$-indescribable in $V$, and it will
	shoot an $m$-club through $S$. If $m > 1$, then this forcing will have slightly better closure properties
	then $T^1(S \cap \gamma)$ (see Remark~\ref{generalized_closure_remark}), which could be useful for
	certain applications, though we have not found any such applications as of yet.
\end{remark}

\section{$\square_1(\kappa)$ can hold nontrivially at a weakly compact cardinal}\label{section_1_sqare_at_a_weakly_compact}

In this section, we will prove Theorem \ref{theorem_1_square_wc}, which implies that if $\kappa$ is $\kappa^+$ weakly compact then the principle $\square_1(\kappa)$ can be forced to hold at a weakly compact cardinal that has many weakly compact cardinals below it. Let us remind the reader that the corresponding relative consistency result was first obtained by Brickhill and Welch \cite{BrickhillWelch} using $L$.

First, we define a forcing to add a generic coherent sequence of $1$-clubs to a Mahlo cardinal $\kappa$.
\begin{definition}\label{definition_one_square_forcing}
Suppose $\kappa$ is a Mahlo cardinal. We define a forcing $\Q(\kappa)$ such that $q$ is a condition in $\Q(\kappa)$ if and only if
\begin{itemize}
	\item $q$ is a sequence with $\dom(q) = \inacc(\kappa) \cap (\gamma^q + 1)$ for some $\gamma^q < \kappa$,
	\item $q(\alpha)=C^q_\alpha$ is a $1$-club subset of $\alpha$ for each $\alpha\in\dom(q)$ and
	\item for all $\alpha,\beta\in\dom(q)$, if $C^q_\beta\cap\alpha\in\Pi^1_0(\alpha)^+$, then $C^q_\alpha=C^q_\beta\cap\alpha$.\footnote{Equivalently, for all $\alpha,\beta\in\dom(q)$, if $\alpha$ is inaccessible and $C_\beta^q\cap\alpha$ is stationary, then $C_\alpha^q=C_\beta^q\cap\alpha$.}
\end{itemize}
The ordering on $\Q(\kappa)$ is defined by letting $p\leq q$ if and only if $p$ is an end extension of $q$.
\end{definition}




\begin{proposition}\label{prop_Q(kappa)stratclosed}
  Suppose $\kappa$ is a Mahlo cardinal. The poset $\Q(\kappa)$ is $\kappa$-strategically closed.
\end{proposition}

\begin{proof}
We describe a winning strategy for Player II in the game $\mathcal G_\kappa(\Q(\kappa))$. We will recursively arrange so that, if $\delta < \kappa$
and $\langle q_\alpha \st \alpha < \delta \rangle$ is a partial play of the game with Player II playing according
to her winning strategy, then, for all limit ordinals $\beta < \delta$, we have $\{\gamma^{q_\alpha} \st \alpha < \beta, ~\alpha \mbox{ even}\}$ is a club in its supremum and,
if $\gamma^{q_\beta}$ is inaccessible, is a subset of $C^{q_\beta}_{\gamma^{q_\beta}}$.
We will also arrange that, for all even successor ordinals $\alpha < \beta < \delta$, $\gamma^{q_\alpha}$ and $\gamma^{q_\beta}$ are
inaccessible cardinals, $C^{q_\beta}_{\gamma^{q_\beta}} \cap \gamma^{q_\alpha} = C^{q_\alpha}_{\gamma^{q_\alpha}}$
and $\{\gamma^{q_\alpha} \st \alpha < \beta, ~ \alpha \text{ even}\} \subseteq C^{q_\beta}_{\gamma^{q_\beta}}$.

We first deal with successor ordinals. Suppose that $\delta < \kappa$ is an even ordinal and $\langle q_\alpha \st \alpha \leq \delta + 1 \rangle$ has been played. Suppose first that $\gamma^{q_\delta}$ is an inaccessible cardinal (in particular, by our recursion hypotheses, this must be the case if $\delta$ is a successor ordinal). In this case, let $\gamma^{q_{\delta + 2}}$ be the least inaccessible cardinal above $\gamma^{q_{\delta + 1}}$ and let $q_{\delta + 2}$ be the condition extending $q_{\delta + 1}$ by setting
  \[ C^{q_{\delta + 2}}_{\gamma^{q_{\delta + 2}}} = C^{q_{\delta}}_{\gamma^{q_\delta}}
    \cup \{\gamma^{q_\delta}\} \cup [\gamma^{q_{\delta + 1}}, \gamma^{q_{\delta + 2}}).
  \]
The fact that $C^{q_\delta}_{\gamma^{q_\delta}} \cup \{\gamma^{q_\delta}\} \subseteq C^{q_\delta + 2}_{\gamma^{q_{\delta + 2}}}$
ensures that the recursion hypothesis is maintained. The set $C^{q_{\delta + 2}}_{\gamma^{q_{\delta + 2}}}$ is stationary
in $\gamma^{q_{\delta + 2}}$ because it contains a tail, and it has all its inaccessible stationary reflection points because
those are $\lesseq \gamma^{q_{\delta}}$. The coherence property holds because we have omitted the interval $(\gamma^{q_\delta}, \gamma^{q_{\delta + 1}})$
from $C^{q_{\delta + 2}}_{\gamma^{q_{\delta + 2}}}$ ensuring that for no $\alpha$ in that interval is $C^{q_{\delta + 2}}_{\gamma^{q_{\delta + 2}}}\cap\alpha$ stationary.
It follows that $q_{\delta + 2}$ is a condition and a valid play for Player II.

If $\gamma^{q_\delta}$ is not inaccessible, then $\delta$ is a limit ordinal (by our recursion hypothesis). In this case, again let $\gamma^{q_{\delta + 2}}$ be the least inaccessible cardinal above $\gamma^{q_{\delta + 1}}$, and define $q_{\delta + 2}$ by setting
  \[
    C^{q_{\delta + 2}}_{\gamma^{q_{\delta + 2}}} = \bigcup_{\substack{\alpha < \delta \\
    \alpha \text{ even}}}C^{q_\delta}_{\gamma^{q_{\alpha + 2}}} \cup \{\gamma^{q_\delta}\}
    \cup [\gamma^{q_{\delta + 1}}, \gamma^{q_{\delta + 2}}).
  \]
A similar argument as above verifies that $q_{\delta + 2}$ is a valid play in the game and maintains our recursion hypotheses.

  Finally, suppose that $\delta < \kappa$ is a limit ordinal and $\langle q_\alpha
  \st \alpha < \delta \rangle$ has been played. Let $\gamma^{q_\delta} =
  \sup\{\gamma^{q_\alpha} \st \alpha < \delta\}$. If $\gamma^{q_\delta}$ is not
  inaccessible, then we can simply set $q_\delta = \bigcup_{\alpha < \delta} q_\alpha$.
  If $\gamma^{q_\delta}$ is inaccessible, then we must additionally define
  $C^{q_\delta}_{\gamma^{q_\delta}}$. We do this by setting
  \[
    C^{q_\delta}_{\gamma^{q_\delta}} = \bigcup_{\substack{\alpha < \delta \\
    \alpha \text{ even}}}C^{q_\delta}_{\gamma^{q_{\alpha + 2}}}.
  \]
  It is easy to verify that this is as desired. The fact that $C^{q_\delta}_{\gamma^{q_\delta}}$
  is stationary in $\gamma^{q_\delta}$  follows from the fact that $\{\gamma^{q_\alpha} \st \alpha < \delta, ~
  \alpha \mbox{ even}\} \subseteq C^{q_\delta}_{\gamma^{q_\delta}}$, so it in
  fact contains a club in $\gamma^{q_\delta}$.
\end{proof}
It follows from Proposition~\ref{prop_Q(kappa)stratclosed} that $\Q(\kappa)$ is $\lt\kappa$-distributive. In particular, if $G\subseteq\Q(\kappa)$ is a generic filter,
then $\bigcup G$ is a coherent sequence of $1$-clubs of length $\kappa$, since $V_\kappa$ remains unchanged.

Next, we define a forcing which will be used to generically thread a coherent sequence of 1-clubs.
\begin{definition}
Suppose that $\vec{C}(\kappa)=\<C_\alpha(\kappa)\st\alpha\in\inacc(\kappa)\>$ is a coherent sequence of $1$-clubs. The poset $\T(\vec{C}(\kappa))$ consists of all conditions $t$ such that
\begin{itemize}
\item $t$ is a $1$-closed bounded subset of $\kappa$ and
\item for every $\alpha<\kappa$, if $t\cap\alpha\in\Pi^1_0(\alpha)^+$, then $C_\alpha(\kappa)=t\cap\alpha$.\footnote{Equivalently, for every inaccessible cardinal $\alpha<\kappa$, if $t\cap\alpha$ is stationary in $\alpha$ then $C_\alpha(\kappa)=t\cap\alpha$.}
\end{itemize}
The ordering on $\T(\vec{C}(\kappa))$ is defined by letting $t\leq s$ if and only if $t$ end-extends $s$.
\end{definition}

\begin{lemma}\label{lemma_thread_forcing}
Suppose $\kappa$ is a regular cardinal and $\vec{C}(\kappa)$ is a coherent sequence of $1$-clubs. Then the poset $\T(\vec{C}(\kappa))$ is $\kappa$-strategically closed.\footnote{Note that the forcing to thread a $\square(\kappa)$-sequence is never $\kappa$-strategically closed.}
\end{lemma}

\begin{proof}
We describe a winning strategy for player II in $\mathcal{G}_\kappa(\T(\vec{C}(\kappa)))$. Player II's strategy at successor ordinal stages can be arbitrary provided that Player II chooses conditions properly extending Player I's previous play.

So let $\delta$ be a limit stage and let $\<t_\alpha\st \alpha<\delta\>$ be the sequence of conditions played at previous stages of the game. Player II then plays $t_\delta=\left(\bigcup_{\alpha<\delta}t_\alpha\right)\cup\{\kappa_{\delta}+1\}$, where $\kappa_\delta=\sup\left(\bigcup_{\alpha<\delta}t_\alpha\right)$. We will also assume recursively that Player II has played according to this strategy successfully at previous limit stages of the game, so that, if $\lambda < \delta$ is a limit ordinal, then $\kappa_\lambda \notin t_\delta$. It remains to show that $t_\delta\in\T(\vec{C}(\kappa))$.

To argue that $t_\delta$ is a $1$-closed subset of $\kappa$, it suffices to see that $t_\delta\cap\kappa_\delta$ is not stationary in $\kappa_\delta$. By our recursive assumption,  $\{\kappa_\lambda\st \lambda<\delta\}$ is a club subset of $\kappa_\delta$ disjoint from $t_\delta$, and hence $t_\delta\cap\kappa_\delta$ is not stationary in $\kappa_\delta$. The coherence condition follows easily.
\end{proof}

\begin{lemma}\label{lemma_TaddsThread}
Suppose $\kappa$ is a regular cardinal and $\vec{C}(\kappa)=\<C_\alpha(\kappa)\st\alpha\in\inacc(\kappa)\>$ is a coherent sequence of $1$-clubs. If $G\subseteq \T(\vec C(\kappa))$ is generic over $V$, then $C_\kappa=\bigcup G$ threads $\vec C(\kappa)$ in $V[G]$.
\end{lemma}
\begin{proof}
By the $\lt\kappa$-distributivity of $\T(\vec C(\kappa))$ and the definition of its conditions, $C_\kappa$ meets the coherence requirements and contains all its inaccessible stationary reflection points. So it remains to check that $C_\kappa$ is stationary.

Fix a club $C\subseteq\kappa$ in $V[G]$ and let $\dot C$ be a $\T(\vec C(\kappa))$-name for $C$. Assume towards a contradiction that $C\cap C_\kappa=\emptyset$. Fix $t_0\in \T(\vec C(\kappa))$ forcing that $\dot C$ is a club and $\dot C\cap \dot C_\kappa=\emptyset$, where $\dot C_\kappa$ is the canonical $\T(\vec C(\kappa))$-name for $C_\kappa$, and let $\beta_0$ be the supremum of $t_0$.
Recursively define a decreasing sequence $\langle t_n \st n < \omega \rangle$ of conditions from $\T(\vec{C}(\kappa))$ as follows, letting $\beta_n$ denote $\sup(t_n)$.
Given $n < \omega$, if $t_n$ is defined, find an ordinal $\alpha_n$ with $\beta_n < \alpha_n < \kappa$
and a condition $t_{n+1} \leq t_n$ such that $t_{n+1} \forces \check{\alpha}_n \in \dot{C}$.
Let $\alpha=\bigcup_{n<\omega}\alpha_n=\bigcup_{n<\omega}\beta_n$, and let $t=\bigcup_{n < \omega} t_n \cup\{\alpha\}$.
Clearly $t$ is a condition in $\T(\vec C(\kappa))$ and $t\forces \alpha\in\dot C \cap \dot C_\kappa$, which is the desired contradiction.
\end{proof}

\begin{theorem}\label{theorem_forcing_1square_at_a_weakly_compact}
Suppose $\kappa$ is weakly compact and the $\GCH$ holds. There is a cofinality-preserving forcing extension in which
\begin{enumerate}
\item for all $\gamma\leq\kappa$, every set $W\in P(\gamma)^V$ which is weakly compact in $V$ remains weakly compact and
\item $\square_1(\kappa)$ holds.
\end{enumerate}
\end{theorem}

\begin{proof}

Define an Easton-support iteration $\<(\P_\alpha,\dot\Q_\beta)\st\alpha\leq\kappa+1, ~ \beta\leq\kappa\>$ as follows.
\begin{itemize}
\item If $\gamma<\kappa$ is Mahlo, let $\dot{\Q}_\gamma=(\Q(\gamma)*\dot{\T}(\vec C(\gamma)))^{V^{\P_\gamma}}$, where $\vec C(\gamma)$ is the generic coherent sequence of 1-clubs of length $\gamma$ added by $\Q(\gamma)$.
\item If $\gamma=\kappa$, let $\dot\Q_\kappa=(\Q(\gamma))^{V^{\P_\kappa}}$.
\item Otherwise, let $\dot\Q_\gamma$ be a $\P_\gamma$-name for trivial forcing.
\end{itemize}

Let $G*H\subseteq\P_\kappa*\dot\Q_\kappa$ be generic over $V$. $V[G*H]$ is our desired model.
Standard arguments using progressive closure of the iteration $\P_\kappa$ together with the $\GCH$ show that cofinalities are preserved in $V[G*H]$.

The argument for the preservation of weakly compact subsets of $\gamma<\kappa$ is similar to and easier than the argument for the preservation of weakly compact subsets of $\kappa$, and we leave it to the reader.

Recall that $\vec{C}(\kappa)=\bigcup H$ is a coherent sequence of $1$-clubs of length $\kappa$. Fix $W\in P(\kappa)^V$ which is weakly compact in $V$. It remains to argue that in $V[G*H]$, $W$ is weakly compact and $\vec{C}(\kappa)$ has no thread.

Fix a set $C\in P(\kappa)^{V[G*H]}$ which is a $1$-club subset of $\kappa$ in $V[G*H]$. We will simultaneously show that $C$ is not a thread through $\vec{C}(\kappa)$ and that $W$ remains weakly compact in $V[G*H]$. Fix $A\in P(\kappa)^{V[G*H]}$ and let $\dot{C},\dot{A},\tau\in H(\kappa^+)^V$ be $\P_{\kappa+1}$-names with $\dot{C}_{G*H}=C$, $\dot{A}_{G*H}=A$ and $\tau_{G*H}=\vec{C}(\kappa)$. Let $M$ be a $\kappa$-model with $W,\dot{C},\dot{A},\tau,\P_{\kappa+1}\in M$. Since $W$ is weakly compact in $V$, there is a $\kappa$-model $N$ and an elementary embedding $j:M\to N$ such that
$\crit(j) = \kappa$ and $\kappa\in j(W)$.


Since $N^{<\kappa}\cap V\subseteq N$, we have, in $N$, $$j(\P_\kappa)\cong\P_\kappa*(\dot{\Q}(\kappa)*\dot{\T}(\vec C(\kappa)))*\dot{\P}_{\kappa,j(\kappa)},$$ where $\dot\P_{\kappa,j(\kappa)}$ is a $\P_{\kappa+1} * \dot{\bb{T}}(\vec{C}(\kappa))$-name for the iteration from $\kappa+1$ to $j(\kappa)$. By Lemma \ref{lemma_thread_forcing}, $\T(\vec C(\kappa))$ is $\kappa$-strategically closed in $N[G*H]$, and hence, using standard arguments, we can build a filter $h\in V[G*H]$ for $\T(\vec C(\kappa))$ which is generic over $N[G*H]$. Let $C_\kappa=\bigcup h$ and notice that $C_\kappa\neq C$ because $C\in N[G*H]$ and $C_\kappa$ is generic over $N[G*H]$. Similarly, we can build a filter  $G'\in V[G*H]$ which is generic for $\P_{\kappa,j(\kappa)}=(\dot \P_{\kappa,j(\kappa)})_{G*H*h}$ over $N[G*H*h]$. Since $j\image G\subseteq G*H*h*G'$, the embedding can be extended to $j:M[G]\to N[\hat{G}]$, where $\hat{G}=G*H*h*G'$.

Let $\Q(\kappa)=(\dot \Q(\kappa))_G$. Working in $N[\hat{G}]$, since $$\vec{C}(\kappa)=\bigcup H=\<C_\alpha(\kappa)\st\alpha\in\inacc(\kappa)\>$$ is a coherent sequence of $1$-clubs and $C_\kappa$ is a thread through $\vec{C}(\kappa)$ by Lemma~\ref{lemma_TaddsThread}, it follows that the function $$q=\<C_\alpha(\kappa)\st\alpha\in\inacc(\kappa)\>\cup\{(\kappa,C_\kappa)\}$$ is a condition in $j(\Q(\kappa))$ below every element of $j\image H$. We may build a filter $\hat{H}\in V[G*H]$ which is generic for $j(\Q(\kappa))$ over $N[\hat{G}]$ with $q\in \hat{H}$. Since $j\image H\subseteq\hat{H}$, it follows that $j$ extends to $j:M[G*H]\to N[\hat{G}*\hat{H}]$. Now $A\in M[G*H]$ and $\kappa\in j(W)$, so $W$ is weakly compact in $V[G*H]$.

It remains to show that $C$ is not a thread through $\vec{C}(\kappa)$. For the sake of contradiction, assume $C$ is a thread through $\vec{C}(\kappa)$. By elementarity we see that in $N[\hat{G}*\hat{H}]$, $$j(\vec{C}(\kappa))=\<\bar{C}_\alpha(j(\kappa))\st\alpha\in\inacc(j(\kappa))\>$$ is a coherent sequence of $1$-clubs. Since $q=\vec{C}(\kappa)\concat\<C_\kappa\>\in \hat{H}$ we have $\bar{C}_\kappa(j(\kappa))=C_\kappa$. Now since $C$ is a thread for $\vec{C}(\kappa)$ in in $M[G*H]$, by elementarity, $j(C)$ is a thread for $j(\vec{C}(\kappa))=\<\bar{C}_\alpha(j(\kappa))\st\alpha\in\inacc(j(\kappa))\>$. Since $\kappa$ is inaccessible in $N[\hat{G}*\hat{H}]$ and $\kappa\in \Tr_0(j(C))$, it follows that $C_\kappa=\bar{C}_\kappa(j(\kappa))=j(C)\cap\kappa=C$, a contradiction.
\end{proof}

\begin{remark}\label{rem_groundmodelGenericBreaksCorrectness}
Observe that in the proof of Theorem \ref{theorem_forcing_1square_at_a_weakly_compact}, if we assume that $\kappa$ is $\Pi^1_2$-indescribable and that the target $N$ of the embedding $j:M\to N$ we start with is $\Pi^1_1$-correct, then the $\kappa$-model $N[G*H]$ from the proof of Theorem~\ref{theorem_forcing_1square_at_a_weakly_compact} is $\Pi^1_1$-correct by Theorem~\ref{theorem_easton_iterations_preserve_correctness} and Corollary~\ref{corollary_closed_forcing_preserves_correctness}. However, the $\kappa$-model $N[G*H*h]$ cannot be $\Pi^1_1$-correct because otherwise we would have shown that, in the extension $V[G*H]$, $\kappa$ is $\Pi^1_2$-indescribable, contradicting Proposition~\ref{proposition_n_square_denies_n_plus_1_indescribability}. Thus, a forcing extension of a $\Pi^1_1$-correct $\kappa$-model, even by a $\kappa$-strategically closed forcing notion, need not be $\Pi^1_1$-correct
if the generic filter is not fully $V$-generic.
\end{remark}

For the next theorem, let us recall what it means for a cardinal $\kappa$ to be $\alpha$-weakly compact, where $\alpha\leq\kappa^+$. Suppose $\kappa$ is a weakly compact cardinal. It is not difficult to see that if sets $X, Y\in P(\kappa)$ are equivalent modulo the ideal $\Pi^1_1(\kappa)$, then their traces $\Tr_1(X)$ and $\Tr_1(Y)$ are equivalent as well. Thus, the trace operation $\Tr_1:P(\kappa)\to P(\kappa)$ leads to a well defined operation $\Tr_1:P(\kappa)/\Pi^1_1(\kappa)\to P(\kappa)/\Pi^1_1(\kappa)$ on the collection $P(\kappa)/\Pi^1_1(\kappa)$ of equivalence classes of subsets of $\kappa$ modulo the ideal $\Pi^1_1(\kappa)$. By taking diagonal intersections at limit ordinals, we can iterate the trace operation on the equivalence classes $\kappa^+$-many times.
To be more precise, fix a sequence $\langle e_\beta \mid \kappa \leq \beta < \kappa^+, ~ \beta \text{ limit} \rangle$, where $e_\beta : \kappa \rightarrow \beta$
is a bijection for all relevant $\beta$. To start, let $\Tr_1^1 = \Tr_1$. Given $\alpha < \kappa^+$, if $\Tr_1^\alpha : P(\kappa)/\Pi^1_1(\kappa) \to P(\kappa)/\Pi^1_1(\kappa)$
has been defined, let $\Tr_1^{\alpha + 1} = \Tr_1 \circ \Tr_1^\alpha$. If $\beta < \kappa$ is a limit ordinal and $\Tr_1^\alpha$ has been defined for all $\alpha < \beta$,
then define $\Tr_1^\beta$ by letting $\Tr_1^\beta([S]) = [\bigcap_{\alpha < \beta} S_\alpha]$, where $S_\alpha$ is a representative
element of $\Tr_1^\alpha([S])$ for all $\alpha < \beta$. Finally, if $\beta$ is a limit ordinal and $\kappa \leq \beta < \kappa^+$,
then let $\Tr_1^\beta([S]) = [\triangle_{\eta < \kappa}S_{e_\beta(\eta)}]$.

It is straightforward to verify that each of these
functions is well-defined and does not depend on our choice of $e_\beta$ for limit $\beta$.
For $\alpha < \kappa^+$, the cardinal $\kappa$ is then said to be $\alpha$-\emph{weakly compact} if $\Tr_1^\alpha([\kappa]) \neq [\emptyset]$,
and $\kappa$ is $\kappa^+$-\emph{weakly compact} if it is $\alpha$-weakly compact for all $\alpha < \kappa^+$.
For more details, the reader is referred to \cite{MR3985624}.

If we start with a $\kappa^+$-weakly compact cardinal $\kappa$ in Theorem \ref{theorem_forcing_1square_at_a_weakly_compact}, then it will remain $\kappa^+$-weakly compact in the extension $V[G*H]$. Because weakly compact subsets of all cardinals $\gamma\leq\kappa$ are preserved to $V[G*H]$, it is easy to show by induction on $\alpha\leq\kappa^+$ that if a set $X$ is in the equivalence class $\Tr^\alpha([\kappa])$ as computed in $V$, then the equivalence class $\Tr^\alpha([\kappa])$ as computed in $V[G*H]$ contains some $Y\supseteq X$. Thus, we get the following.

\begin{theorem_1_square_wc}
If $\kappa$ is $\kappa^+$-weakly compact and $\GCH$ holds then there is a cofinality preserving forcing extension in which
\begin{enumerate}
\item $\kappa$ remains $\kappa^+$-weakly compact and
\item $\square_1(\kappa)$ holds.
\end{enumerate}
\end{theorem_1_square_wc}

Next we will show that, if $\kappa$ is Mahlo, one can characterize precisely when $\square_1(\kappa)$ holds after forcing with $\Q(\kappa)$.
Notice that, if there is a stationary subset of $\kappa$ that does not reflect at an inaccessible cardinal (i.e., if $\Refl_0(\kappa)$) fails, then $\square_1(\kappa)$
must fail, since any such non-reflecting stationary set of $\kappa$ would then be a thread through
any coherent sequence of $1$-clubs of length $\kappa$. We will see in Theorem~\ref{theorem_reflection_characterization}
that $\Refl_0(\kappa)$ holding in the extension by $\Q(\kappa)$ is in fact sufficient for $\square_1(\kappa)$ to hold.
First, we need the following general proposition. Recall that $\Refl_n(\kappa)$ holds if and only if $\kappa$ is $\Pi^1_n$-indescribable and,
for every $\Pi^1_n$-indescribable subset $S$ of $\kappa$, there is an $\alpha<\kappa$ such that $S\cap \alpha$ is $\Pi^1_n$-indescribable.

\begin{proposition}\label{proposition_reflection}
Fix $n < \omega$. If $\kappa$ is a cardinal, $\Refl_n(\kappa)$ holds and $S\in\Pi^1_n(\kappa)^+$, then the set
\[T=\{\alpha<\kappa\st (\text{$S\cap\alpha\in\Pi^1_n(\alpha)^+$}) \land (\text{$\Refl_n(\alpha)$ fails}) \}\]
is $\Pi^1_n$-indescribable.
\end{proposition}
\begin{proof}
We proceed by induction on $\kappa$. Suppose the proposition holds for all cardinals $\alpha < \kappa$, $\Refl_n(\kappa)$ holds and
$S \in \Pi^1_n(\kappa)^+$. It suffices to show that $T\cap C\neq\emptyset$ for every $n$-club subset $C$ of $\kappa$.
Fix an $n$-club set $C$ and note that $S\cap C$ is $\Pi^1_n$-indescribable. Thus, by $\Refl_n(\kappa)$, there is some $\alpha_0<\kappa$ such that $S\cap C\cap\alpha_0\in \Pi^1_n(\alpha_0)^+$. It follows that $\alpha_0\in \Tr_n(S)$, but also $\alpha_0\in C$ because $C$ contains all of its $\Pi^1_{n-1}$-reflection points. If $\alpha_0\in T$, we have shown that $T\cap C\neq\emptyset$.
So suppose that $\alpha_0\notin T$, so $\Refl_{n}(\alpha_0)$ holds. We can now appeal to the inductive hypothesis
at $\alpha_0$, applied to the $\Pi^1_n$-indescribable set $S \cap \alpha_0$ and the $n$-club $C \cap \alpha_0$,
to find a cardinal $\alpha_1 \in T \cap C$.
\end{proof}


\begin{theorem}\label{theorem_reflection_characterization}
  Suppose $\kappa$ is Mahlo and $p \in \Q(\kappa)$. The following are equivalent:
  \begin{enumerate}
    \item $p \Vdash_{\Q(\kappa)} \Refl_0(\kappa)$
    \item $p \Vdash_{\Q(\kappa)} \square_1(\kappa)$
  \end{enumerate}
\end{theorem}

\begin{proof}
  The implication $(2) \Rightarrow (1)$ follows immediately from the observation
  that a stationary subset of $\kappa$ that does not reflect at any inaccessible
  cardinal is a thread through any putative $\square_1(\kappa)$-sequence.

  We now show $(1) \Rightarrow (2)$. Suppose for the sake of contradiction
  that $p \Vdash_{\Q(\kappa)} \Refl_0(\kappa)$ and there is $p_1 \leq_{\Q(\kappa)} p$ such that
  $p_1 \Vdash_{\Q(\kappa)} \neg \square_1(\kappa)$. In particular, $p_1$ forces
  that $\bigcup \dot{G}$ is not a $\square_1(\kappa)$-sequence, so there is a
  $\Q(\kappa)$-name $\dot{C}$ that is forced by $p_1$ to be a thread through
  $\bigcup \dot{G}$.

  Let $G$ be $\Q(\kappa)$-generic over $V$ with $p_1 \in G$, and move to $V[G]$.
  Let $C = \dot{C}_G$. Since $C$ is stationary
  in $\kappa$ and $\Refl_0(\kappa)$ holds, Proposition \ref{proposition_reflection} implies
  that there are stationarily many inaccessible $\lambda < \kappa$ such that
  $C$ reflects at $\lambda$ and $\Refl_0(\lambda)$ fails.

	Next, observe that every sequence of elements of $G$ of size less than $\kappa$ has a lower bound in $G$.
	Suppose that $\beta < \kappa$, and fix in $V[G]$ a sequence $\vec p=\la p_\xi\st \xi<\beta\ra$ of elements of $G$.
	The sequence $\vec p$ must be in $V$ by the $\lt\kappa$-distributivity of $\Q(\kappa)$, and so there is a
	condition $p\in G$ forcing that $\vec p$ is contained in $G$. But then $p$ is a lower bound for $\vec p$.
	Observe also that, for all $\gamma < \kappa$, the initial segment $C^{(\gamma)}=C\cap\gamma$ of $C$ is in $V$.

Now, in $V[G]$, we build a strictly decreasing sequence of conditions $\langle q_\alpha \st \alpha < \kappa \rangle$ from $G$ such that
  \begin{enumerate}
    \item $q_0 = p_1$,
    \item $\{\gamma^{q_\alpha} \st \alpha < \kappa\}$, the set of suprema of the domains of the conditions, is a club and
    \item for all $\alpha < \kappa$, $q_{\alpha + 1} \Vdash_{\Q(\kappa)} \dot{C}
      \cap \gamma^{q_\alpha} = \check C^{(\alpha)}.$
  \end{enumerate}
We can ensure that (2) holds as follows. At a limit stage $\lambda < \kappa$, given that we have already constructed $\la q_\alpha\st\alpha<\lambda\ra$, we know that there is some $q\in G$ below our sequence. So we let $\gamma_\lambda=\bigcup_{\alpha<\lambda}\gamma_\alpha$ and take $q_\lambda=q\restrict\gamma_\lambda+1$.

Thus, we can find an inaccessible cardinal $\lambda$ such that $\lambda = \gamma^{q_\lambda}$, $C$ reflects at $\lambda$, and $\Refl_0(\lambda)$ fails. Since $\Refl_0(\lambda)$ fails (in $V[G]$ and hence also in $V$, since forcing with $\Q(\kappa)$ did not add any bounded subsets to $\kappa$), we can fix in $V$ a stationary $C_\lambda \subseteq \lambda$ that is different from $C^{(\lambda)}=C \cap \lambda$ and that does not reflect at any inaccessible cardinal below $\lambda$.
  Now form a condition $q_\lambda^* \in \mathbb{Q}$ with $\gamma^{q_\lambda^*} =
  \gamma^{q_\lambda} = \lambda$ by letting $q_\lambda^* \restriction \lambda =
  \bigcup_{\alpha < \lambda} q_\alpha$ and $C^{q_\lambda^*}_\lambda = C_\lambda$.
  This is easily seen to be a valid condition, because everything needed to
  construct it is in $V$ and since $C_\lambda$ does not reflect at any
  inaccessible cardinal. Since $q_\lambda^* \leq_{\Q(\kappa)} q_\alpha$ for all
  $\alpha < \lambda$, we have $$q_\lambda^* \Vdash_{\Q(\kappa)}\dot{C} \cap \lambda
  = \check C^{(\lambda)}.$$ In particular, since $C^{(\gamma)}=C \cap \lambda$ is stationary in
  $\lambda$, and since $q_\lambda^*$ extends $p_1$ and thus forces that
  $\dot{C}$ is a thread through $\bigcup \dot{G}$, it must be the case that
  $q_\lambda^*$ forces that the $\lambda$-th entry in $\bigcup \dot{G}$
  is $C^{(\lambda)}$. However, $q_\lambda^*$ forces the $\lambda$-th
  entry in $\bigcup \dot{G}$ to be $C_\lambda$, which is different from
  $C \cap \lambda$. This gives the desired contradiction.
\end{proof}

\begin{remark}
	Since the weak compactness of $\kappa$ implies $\Refl_0(\kappa)$, by Theorem \ref{theorem_reflection_characterization} it follows that in the proof of Theorem \ref{theorem_forcing_1square_at_a_weakly_compact}, in order to show that $\square_1(\kappa)$ holds in $V[G*H]$ it suffices to show that $\kappa$ remains weakly compact.
\end{remark}

\section{Consistency of $\square_1(\kappa)$ with $\Refl_1(\kappa)$}\label{sec_squareAndReflection}
In this section, we will show that the principle $\square_1(\kappa)$ is consistent with $\Refl_1(\kappa)$. First, we will need a lemma showing that we can force the existence of a fast function while preserving $\Pi^1_2$-indescribability.

The fast function forcing $\F_\kappa$, introduced by Woodin, consists of conditions that are partial functions $p\from\kappa\to\kappa$ such that for every $\gamma\in\dom(p)$,
the following conditions hold:
\begin{itemize}
	\item $\gamma$ is inaccessible,
	\item $p \image \gamma \subseteq \gamma$, and
	\item $|p \restrict \gamma| < \gamma$.
\end{itemize}
The union $f\from\kappa\to\kappa$ of a generic filter for $\F_\kappa$ is called a \emph{fast function}. Let $\F_{[\gamma,\kappa)}$ denote the subset of $\F_\kappa$ consisting of conditions $p$ with $\dom(p)\subseteq [\gamma,\kappa)$ and observe that $\F_{[\gamma,\kappa)}$ is $\lesseq\gamma$-closed. It is not difficult to see that for any condition $p\in\F_\kappa$ and $\gamma\in\dom(p)$, the forcing $\F_\kappa$ factors below $p$ as
\[
	\F_\gamma\restrict p\cong \F_\kappa\restrict (p\restrict \gamma)\times \F_{[\gamma,\kappa)}\restrict ( p \restrict [\gamma,\kappa)).
\]
\begin{lemma}\label{fast_function_lemma}
Suppose $\kappa$ is $\Pi^1_n$-indescribable. In a generic extension $V[f]$ by fast function forcing, $\kappa$ remains $\Pi^1_n$-indescribable and the fast function $f$ has the following property. For every $A \in H(\kappa^+)$ and $\alpha<\kappa^+$, there are a $\kappa$-model $M$ with $f,A\in M$, a $\Pi^1_{n-1}$-correct $\kappa$-model $N$ and an elementary embedding $j:M\to N$ with critical point $\kappa$ such that $j(f)(\kappa)=\alpha$ and $j,M\in N$.
\end{lemma}
\begin{proof}
The cardinal $\kappa$ remains inaccessible in $V[f]$ because for unboundedly many inaccessible $\alpha<\kappa$, there is a condition $p\in G$ with $\alpha\in\dom p$, so
$\mathbb F_\kappa$ below $p$ factors with a first factor of size $\alpha$ and a second factor that is $\lesseq\alpha$-closed.

Fix $A\in H(\kappa^+)^{V[f]}$ and $\alpha<\kappa^+$ (note that $V$ and $V[f]$ have the same $\kappa^+$). Let $\dot A$ be an $\F_{\kappa}$-name for $A$ and let $B\subseteq\kappa$ code $\alpha$. By Theorem~\ref{theorem_hauser}~(4), there are a $\kappa$-model $M$ with $\F_\kappa,\dot A,B\in M$, a $\Pi^1_{n-1}$-correct $\kappa$-model $N$ and an elementary embedding $j:M\to N$ with critical point $\kappa$ such that $j,M\in N$. We will lift $j$ to $M[G]$. Let $p=\la {\kappa,\alpha}\ra$ be a condition in $j(\F_{\kappa})$. Below $p$, $j(\F_\kappa)$ factors as $j(\F_\kappa)\restrict p\cong \F_\kappa\times \F_{[\kappa,j(\kappa))}\restrict p$, where the second factor is $\lesseq\kappa$-closed in $N$.  In $V$, we can build an $N$-generic function $f'$ for $\F_{[\kappa,j(\kappa))}$ containing $p$, and so $f\times f'$ is $N$-generic for $j(\F_\kappa)$. Thus, we can lift $j$ to $j:M[f]\to N[f][f']$, and clearly $M[f]$ and $j$ are in $N[f][f']$.

It remains to verify that $M[f]$ is a $\kappa$-model and $N[f][f']$ is a $\Pi^1_{n-1}$-correct $\kappa$-model. The argument to show that $M[f]$ is a $\kappa$-model in $V[f]$ will be more involved than usual because, as $\F_\kappa$ is not $\kappa$-c.c., we cannot apply the generic closure criterion. Fixing $\beta<\kappa$, we will show that $M[f]^\beta\subseteq M[f]$ in $V[f]$. By density, there is an inaccessible cardinal $\alpha>\beta$ and a condition $p=\la \{\gamma,\delta\}\ra\in G$ such that $\gamma<\alpha<\delta$. Below $p$, $\F_\kappa$ factors as $\F_\gamma\times \F_{(\delta,\kappa)}$ and $f$ factors as $f_\gamma\times f_{(\delta,\kappa)}$. Since $\F_\gamma$ clearly has the $\alpha$-c.c., by the generic closure criterion, $M[f_\gamma]^\beta\subseteq M[f_\gamma]$ in $V[f_\gamma]$. Also, since $\F_{(\delta,\kappa)}$ is $\lesseq\alpha$-closed, $M[f_\gamma]^\beta\subseteq M[f_\gamma]$ in $V[f]$. Finally, by the ground closure criterion, $M[f_\gamma][f_{(\delta,\kappa)}]^\beta\subseteq M[f_\gamma][f_{(\delta,\kappa)}]$ in $V[f]$. The same argument shows that $N[f]$ is a $\kappa$-model in $V[f]$, and therefore, $N[f][f']$ is a $\kappa$-model as well. To show that $N[f]$ is $\Pi^1_{n-1}$-correct, we argue essentially as in the proof of Theorem~\ref{theorem_easton_iterations_preserve_correctness}. The arguments in that proof go through noting only that $V_\kappa^{V[f]}=V_\kappa^{N[f]}=V_\kappa[f]$ and $\F_\kappa\subseteq V_\kappa$. Finally, the model $N[f][f']$ must also be $\Pi^1_{n-1}$-correct because the tail forcing $\F_{[\kappa,j(\kappa))}$ does not add any subsets to $V_\kappa[f]$ by closure.
\end{proof}

It is not difficult to see that once we have a fast function, we also get a weak Laver function \cite{Hamkins:AClassOfStrongDiamondPrinciples}.

\begin{lemma}\label{lemma_laver_function}
Suppose $\kappa$ is $\Pi^1_n$-indescribable. In the generic extension $V[f]$ by fast function forcing, there is a function $\ell\from\kappa\to V_\kappa$ satisfying the following property. For all $A,B\in H(\kappa^+)^{V[f]}$, there are a $\kappa$-model $M$ with $\ell,A,B\in M$, a $\Pi^1_{n-1}$-correct $\kappa$-model $N$ and an elementary embedding $j:M\to N$ with critical point $\kappa$ such that $j(\ell)(\kappa)=B$ and $j,M\in N$.
\end{lemma}
\begin{proof}
Fix any bijection $b:\kappa\to V_\kappa$ in $V$. In $V[f]$, define $\ell\from \kappa\to V_\kappa$ by letting $\ell(\gamma)=b(f(\gamma))_{f\restrict\gamma}$ provided that
$f \restrict \gamma$ is $\F_\gamma$-generic over $V$ and $b(f(\gamma))$ is an $\F_\gamma$-name. By adapting the proof of Lemma~\ref{fast_function_lemma}, we verify
that $\ell$ has the desired properties as follows. Working in $V[f]$, fix $A,B\in H(\kappa^+)^{V[f]}$ and let $\dot{\ell},\dot{A},\dot{B}$ be nice $\F_\kappa$-names for $\ell,A$ and $B$ respectively. By Theorem \ref{theorem_hauser} (4), there are a $\kappa$-model $M$ with $\dot{\ell},\dot{A},\dot{B},\F_\kappa,b\in M$, a $\Pi^1_{n-1}$-correct $\kappa$-model $N$ and an elementary embedding $j:M\to N$ with critical point $\kappa$ such that $j,M\in N$. Since $\F_\kappa$ is $\kappa^+$-c.c., we can assume without loss of generality that $\dot{B}\in j(V_\kappa)$. By elementarity $j(b):j(\kappa)\to j(V_\kappa)$ is a bijection, and thus there is some ordinal $\alpha<j(\kappa)$ such that $j(b)(\alpha)=\dot B$. As in the proof of Lemma~\ref{fast_function_lemma}, we may lift $j$ to $j:M[f]\to N[f][f']$ such that $j(f)(\kappa)=\alpha$. Now we have $j(\ell)(\kappa)=j(b)(j(f)(\kappa))_{j(f)\restrict\kappa}=j(b)(\alpha)_{f}=\dot{B}_{f}=B$. Now one may prove that $M[f]$ is a $\kappa$-model and $N[f][f']$ is a $\Pi^1_{n-1}$-correct $\kappa$-model exactly as in the proof of Lemma~\ref{fast_function_lemma}.
\end{proof}

\begin{theorem_1_square_and_reflection}
Suppose $\kappa$ is $\Pi^1_2$-indescribable and $\GCH$ holds. Then there is a cofinality-preserving forcing extension $V[G]$ in which
\begin{enumerate}
\item $\square_1(\kappa)$ holds,
\item $\Refl_1(\kappa)$ holds and
\item $\kappa$ is $\kappa^+$-weakly compact.
\end{enumerate}
\end{theorem_1_square_and_reflection}

\begin{proof}
By passing to an extension with a fast function, we can assume without loss of generality that there is a function $\ell\from\kappa\to V_\kappa$ such that for any $A,B\in H(\kappa^+)$ there are a $\kappa$-model $M$ with $\ell,A,B\in M$, a $\Pi^1_1$-correct $\kappa$-model $N$ and an elementary embedding $j:M\to N$ with critical point $\kappa$ such that $j(\ell)(\kappa)=B$.

Let $\<\P_\alpha,\dot{\Q}_\beta\st\alpha\leq\kappa, ~ \beta<\kappa\>$ be the Easton-support iteration defined as follows.
\begin{itemize}
\item If $\alpha<\kappa$ is inaccessible and $\ell(\alpha)$ is a $\P_\alpha$-name for an $\alpha$-strategically closed, $\alpha^+$-c.c. forcing notion, then $\dot{\Q}_\alpha=\ell(\alpha)$.
\item Otherwise, $\dot{\Q}_\alpha$ is a $\P_\alpha$-name for trivial forcing.
\end{itemize}

Let $G$ be generic for $\P_\kappa$ over $V$. In $V[G]$, we define a 2-step iteration $$\Q_\kappa=\Q_{\kappa,0}*(\dot{\Q}_{\kappa,1}\times\dot{\Q}_{\kappa,2})$$ as follows.
\begin{itemize}
\item $\Q_{\kappa,0}$ is the forcing to add a $\square_1(\kappa)$-sequence from Definition \ref{definition_one_square_forcing}.
\item $\dot{\Q}_{\kappa,2}$ is a $\Q_{\kappa,0}$-name for the forcing $\T(\vec{C}(\kappa))$ to thread the generic $\square_1(\kappa)$-sequence.
\item $\dot{\Q}_{\kappa,1}$ is a $\Q_{\kappa,0}$-name for an iteration $\<\R_\eta,\dot{\S}_\xi\st \eta\leq\kappa^+, ~ \xi<\kappa^+\>$ with supports of size ${<}\kappa$ defined as follows. For each $\eta<\kappa^+$, a $\Q_{\kappa,0}*\dot{\R}_\eta$-name $\dot{S}_\eta$ is chosen for a stationary subset of $\kappa$ such that
\[\forces_{\Q_{\kappa,0}*(\dot{\R}_\eta\times\dot{\Q}_{\kappa,2})}\text{``there is a $1$-club in $\kappa$ disjoint from $\dot{S}_\eta$'',}\]
and then $\dot{\S}_\eta$ is a $\Q_{\kappa,0}*\dot{\R}_\eta$-name for the forcing $T^1(\kappa\setminus \dot{S}_\eta)$ to shoot a $1$-club through the complement of $\dot{S}_\eta$.
\end{itemize}
Notice that $\P_\kappa$ is $\kappa$-c.c. and preserves $\GCH$ and, in $V[G]$, the forcing $\Q_\kappa$ is $\kappa$-strategically closed and $\kappa^+$-c.c.. By standard chain condition arguments and bookkeeping, we can ensure that in $V^{\P_\kappa*\dot{\Q}_{\kappa,0}*\dot{\Q}_{\kappa,1}}$, if $S\subseteq\kappa$ is stationary and
\[\forces_{\Q_{\kappa,2}}\text{``there is a $1$-club in $\kappa$ disjoint from $\check{S}$'',}\]
then there is already a $1$-club in $\kappa$ disjoint from $S$.

Let $H=h_0*(h_1\times h_2)$ be generic for $\Q_\kappa$ over $V[G]$. Our desired model will be $V[G*h_0*h_1]$. We must show that in $V[G*h_0*h_1]$, $\kappa$ is $\kappa^+$-weakly compact, $\Refl_1(\kappa)$ holds and $\square_1(\kappa)$ holds.

In order to show that $\kappa$ is $\kappa^+$-weakly compact in $V[G*h_0*h_1]$, we will first prove the following.

\begin{claim}\label{claim_pi12_indescribable}
$\kappa$ is $\Pi^1_2$-indescribable in $V[G*H]$.
\end{claim}

\begin{proof}
Fix $A\in P(\kappa)^{V[G*H]}$. We must find a $\kappa$-model $M$ with $A\in M$, a $\Pi^1_1$-correct $\kappa$-model $N$ and an elementary embedding $j:M\to N$ with critical point $\kappa$.

Let $\dot{A}\in V$ be a $\P_\kappa*\dot{\Q}_\kappa$-name for $A$. Since $\P_\kappa*\dot{\Q}_{\kappa,0}*(\dot{\Q}_{\kappa,1}\times\dot{\Q}_{\kappa,2})$ has the $\kappa^+$-c.c., we can fix $\eta<\kappa^+$ such that $\dot{A}$ is a $\P_\kappa*\dot{\Q}_{\kappa,0}*(\dot{\R}_\eta\times\dot{\Q}_{\kappa,2})$-name. Moreover, we can assume that $\dot{A}$, $\P_\kappa$ and $\dot{\Q}_{\kappa,0} * (\dot{\R}_\eta \times \dot{\Q}_{\kappa,2})$ are in $H(\kappa^+)$. For $\eta<\kappa^+$, let $h_1\restrict\eta$ be the generic for $\R_\eta$ induced by $h_1$.

By Proposition \ref{lemma_laver_function}, there are a $\kappa$-model $M$ with $\ell,\P_\kappa,\dot{A},\dot{\Q}_{\kappa,0}*(\dot{\R}_\eta\times\dot{\Q}_{\kappa,2})\in M$, a $\Pi^1_1$-correct $\kappa$-model $N$ and an elementary embedding $j:M\to N$ with critical point $\kappa$ such that $j(\ell)(\kappa)=\dot{\Q}_{\kappa,0}*(\dot{\R}_\eta\times\dot{\Q}_{\kappa,2})$ and $j,M\in N$. Without loss of generality we may additionally assume that $M\models |\eta|=\kappa$ since a bijection witnessing this can easily be placed into such a $\kappa$-model.

Notice that $j(\P_\kappa)$ is an Easton-support iteration in $N$ of length $j(\kappa)$ and
\[j(\P_\kappa)\cong\P_\kappa*(\dot{\Q}_{\kappa,0}*(\dot{\R}_\eta\times\dot{\Q}_{\kappa,2}))*\dot{\P}_{\kappa,j(\kappa)}\] by our choice of $j(l)(\kappa)$.
By Theorem \ref{theorem_easton_iterations_preserve_correctness}, $N[G]$ is $\Pi^1_1$-correct in $V[G]$, and by Corollary \ref{corollary_closed_forcing_preserves_correctness}, $N[G*h_0*(h_1\restrict\eta\times h_2)]$ is $\Pi^1_1$-correct in $V[G*h_0*(h_1\restrict\eta\times h_2)]$. Hence $N[G*h_0*(h_1\restrict\eta\times h_2)]$ is $\Pi^1_1$-correct in $V[G*h_0*(h_1\times h_2)]$ by Corollary~\ref{cor_kappaStatClosedPreservesPi11}.

Since $(\dot \P_{\kappa,j(\kappa)})_{G*h_0*(h_1\restrict\eta\times h_2)}=\P_{\kappa,j(\kappa)}$ is $\kappa$-strategically closed in\break $N[G*h_0*(h_1\restrict\eta\times h_2)]$ and since $N[G*h_0*(h_1\restrict\eta\times h_2)]$ is a $\kappa$-model in $V[G*H]$, we can build a filter $G'_{\kappa,j(\kappa)}$ which is generic for $\P_{\kappa,j(\kappa}$ over $N[G*h_0*(h_1\restrict\eta\times h_2)]$. Since $j\restrict G$ is the identity function, it follows that $j\image G\subseteq \hat{G}=_{\defn}G*h_0*(h_1\restrict\eta\times h_2)*G_{\kappa,j(\kappa)}$, and thus $j$ lifts to $j:M[G]\to N[\hat{G}]$.

Let $\vec{C}=\<C_\alpha\st\alpha\in\inacc(\kappa)\>$ be the generic $\square_1(\kappa)$-sequence added by $h_0$, and let $T$ be the thread added by $h_2\subseteq\Q_{\kappa,2}=\T(\vec{C}(\kappa))$. By Lemma~\ref{lemma_TaddsThread}, $T$ is a $1$-club in $N[G*h_0*(h_1\restrict\eta\times h_2)]$. Let $p_0=\vec{C}\cup\{(\kappa,T)\}$. Then $p_0\in N[\hat{G}]$ and $p_0\in j(\Q_{\kappa,0})$. Moreover, $p_0\leq_{j(\Q_{\kappa,0})} j(q)$ for all $q\in h_0$. By the strategic closure of $j(\Q_{\kappa,0})$ and the fact that $N[\hat{G}]$ is a $\kappa$-model in $V[G*H]$, we can build a filter $p_0\in\hat{h}_0\subseteq j(\Q_{\kappa,0})$ which is generic over $N[\hat{G}]$. Thus, $j$ extends to $j:M[G*h_0]\to N[\hat{G}*\hat{h}_0]$.

Similarly, in $N[\hat{G}*\hat{h}_0]$, the set $p_2=T\cup\{\kappa\}$ is a condition in $j(\Q_{\kappa,2})$ and $p_2\leq_{j(\Q_{\kappa,2})} j(q)$ for all $q\in h_{2}$. Again, since $j(\Q_{\kappa,2})$ is $\kappa$-strategically closed in $N[\hat{G}*\hat{h}_0]$, which is a $\kappa$-model in $V[G*H]$, we can build a filter $p_2\in \hat{h}_2\subseteq j(\Q_{\kappa,2})$ which is generic for $j(\Q_{\kappa,2})$ over $N[\hat{G}*\hat{h}_0]$, and lift $j$ to $$j:M[G*h_0*h_2]\to N[\hat G*\hat h_0*\hat h_2].$$

Now we lift the embedding through $h_1\restrict\eta$. Let $\R_\eta=(\dot \R_\eta)_{G*h_0}$. By elementarity, $j(\R_\eta)$ is an iteration of length $j(\eta)$ with supports of size less than $j(\kappa)$. For each $\xi<\eta$, $\dot{\S}_{j(\xi)}=j(\dot{\S}_\xi)$ is, in $N[\hat{G}*\hat{h}_0*\hat{h}_2]$, a $j(\R_\xi)=\R_{j(\xi)}$-name for the forcing to shoot a $1$-club disjoint from $j(\dot{S}_\xi)$. For all $\xi<\eta$, let $$D_\xi=\bigcup\{(p(\xi))_{h_1\restrict \xi}\st p\in h_1 \text{ and } \xi\in\dom(p)\}$$ and note that $D_\xi$ is a $1$-club subset of $\kappa$ in $N[\hat{G}*\hat{h}_0*\hat{h}_2]$ because the forcing after $\R_{\xi+1}$ is $\kappa$-strategically closed and therefore cannot affect $\Pi^1_1$-truths by Lemma~\ref{lemma_strategic}. Since $h_1\restrict\eta,j\in N[\hat{G}*\hat{h}_0*\hat{h}_2]$, we can define a function $p^*\in N[\hat{G}*\hat{h}_0*\hat{h}_2]$ such that $\dom(p^*)=j\image\eta$ by letting $p^*(j(\xi))$ be a $j(\R_\xi)$-name for $D_\xi\cup\{\kappa\}$ for all $\xi<\eta$. In order to verify that $p^*\in j(\R_\eta)$, we must show that for all $\xi<\eta$, $p^*\restrict j(\xi)\forces_{j(\R_\xi)}p^*(j(\xi))\cap j(\dot{S}_\xi)=\emptyset$.

Suppose this is not the case, and let $\xi<\eta$ be the minimal counterexample. It follows that $p^*\restrict j(\xi)\in j(\R_\xi)$ and, for all $p\in h_1\restrict\xi$ we have $p^*\restrict j(\xi)\leq j(p)$. By assumption, $$p^*\restrict j(\xi)\not\forces_{j(\R_\xi)}p^*(j(\xi))\cap j(\dot{S}_\xi)=\emptyset$$ and thus we may let $p^{**}\leq_{j(\R_\xi)} p^*\restrict j(\xi)$ be such that $$p^{**}\forces_{j(\R_\xi)}p^*(j(\xi))\cap j(\dot{S}_\xi)\neq\emptyset.$$ Since $j(\R_\xi)$ is sufficiently strategically closed, we can build a filter $\hat{h}_1\subseteq j(\R_\xi)$ in $V[G*H]$ which is generic over $N[\hat{G}*\hat{h}_0*\hat{h}_2]$ with $p^{**}\in \hat{h}_1$ and lift to \[j:M[G*h_0*h_2*(h_1\restrict\xi)]\to N[\hat{G}*\hat{h}_0*\hat{h}_2*\hat{h}_1].\]
It follows that in $N[\hat{G}*\hat{h}_0*\hat{h}_2*\hat{h}_1]$ we have $(D_\xi\cup\{\kappa\})\cap j(S_\xi)\neq\emptyset$, where $S_\xi=(\dot S_\xi)_{h_0\restrict \xi}$. Since $j(S_\xi)\cap\kappa=S_\xi$, we know that $D_\xi\cap j(S_\xi)=\emptyset$, so it must be the case that $\kappa\in j(S_\xi)$. However, in $M[G*h_0*(h_1\restrict\xi)]$, we have $$\forces_{\Q_{\kappa,2}}\text{``there is a $1$-club in $\kappa$ disjoint from $\check S_\xi$''.}$$ Therefore, we can fix such a $1$-club $E$ in $M[G*h_0*h_2*(h_1\restrict\xi)]$. Note that $E$ is actually stationary because $M[G*h_0*h_2*(h_1\restrict \xi)]$ is $\Pi^1_1$-correct by Theorem~\ref{theorem_easton_iterations_preserve_correctness} and Corollary~\ref{corollary_closed_forcing_preserves_correctness}. But then $\kappa\in j(E)$ since in $N[\hat{G}*\hat{h}_0*\hat{h}_2*\hat{h}_1]$, $j(E)$ is $1$-club in $j(\kappa)$ and $j(E)\cap\kappa=E$ is stationary in $\kappa$. Thus $\kappa\in j(E)\cap j(S_\xi)=\emptyset$, a contradiction.

Thus, $p^*\in j(\R_\eta)$ and we can build a filter $p^*\in \hat{h}_1$ in $V[G*H]$ which is generic over $N[\hat{G}*\hat{h}_0*\hat{h}_2]$. This implies that the embedding lifts to
\[j:M[G*h_0*h_2*(h_1\restrict\eta)]\to N[\hat{G}*\hat{h}_0*\hat{h}_2*\hat{h}_1].\]
As we argued above, $N[G*h_0*h_2*(h_1\restrict\eta)]$ is $\Pi^1_1$-correct in $V[G*H]$, and since the forcing
$$\P_{(\kappa,j(\kappa))}*j(\dot \Q_{\kappa,0})*(j(\dot \Q_{\kappa,2})\times j(\dot \R_\eta))$$
is ${\leq}\kappa$-distributive, it follows that $N[\hat{G}*\hat{h}_0*\hat{h}_2*\hat{h}_1]$ is $\Pi^1_1$-correct in $V[G*H]$. Since $A=\dot{A}_{G*h_0*(h_1\restrict\eta\times h_2)}\in M[G*h_0*h_2*(h_1\restrict\eta)]$, this shows that $\kappa$ is $\Pi^1_2$-indescribable in $V[G*H]$.
\end{proof}

Now let us argue that $\kappa$ is $\kappa^+$-weakly compact in $V[G*h_0*h_1]$. Fix $\zeta<\kappa^+$. We must argue that $\Tr_1^\zeta([\kappa])^{V[G*h_0*h_1]} \neq [\emptyset]$. Since $\kappa$ is $\Pi^1_2$-indescribable in $V[G*h_0*(h_1\times h_2)]$ by Claim \ref{claim_pi12_indescribable}, and since $\Q_{\kappa,2}$ is $\kappa$-strategically closed, it follows that $\Tr_1^\zeta([\kappa])^{V[G*h_0*(h_1\times h_2)]}=[S]$, where $S\in V[G*h_0*h_1]$ is $\Pi^1_2$-indescribable in $V[G*h_0*(h_1\times h_2)]$. It follows that $S$ is weakly compact in $V[G*h_0*h_1]$ by Proposition~\ref{proposition_nonresurection}, and clearly $\Tr_1^\zeta([\kappa])^{V[G*h_0*h_1]}=[S]$. Thus, $\kappa$ is $\kappa^+$-weakly compact in $V[G*h_0*h_1]$.

We next argue that $\Refl_1(\kappa)$ holds in $V[G*h_0*h_1]$. Fix a weakly compact set $S\subseteq\kappa$ in $V[G*h_0*h_1]$. Since $S$ intersects every $1$-club in $\kappa$, our construction of $\Q_{\kappa,1}$ implies that there is $p\in\Q_{\kappa,2}$ such that
\[
	p\forces_{\Q_{\kappa,2}} ``\text{there is no 1-club in }\kappa\text{ disjoint from }\check{S}".
\]
Let $g_{2}\subseteq \Q_{\kappa,2}$ be generic over $V[G*h_0*h_1]$ with $p\in g_{2}$. By the proof of Claim~\ref{claim_pi12_indescribable}, $\kappa$ is $\Pi^1_2$-indescribable in $V[G*h_0*h_1*g_{2}]$. Therefore, in $V[G*h_0*h_1*g_{2}]$, $\Refl_1(\kappa)$ holds and $S$ is a weakly compact subset of $\kappa$, and thus there is some $\alpha<\kappa$ such that $S\cap\alpha$ is a weakly compact subset of $\alpha$. But $V[G*h_0*h_1*g_{2}]$ and $V[G*h_0*h_1]$ have the same $V_\kappa$, so $S\cap\alpha$ is a weakly compact subset of $\alpha$ in $V[G*h_0*h_1]$. Thus, $\Refl_1(\kappa)$ holds in $V[G*h_0*h_1]$.

Finally, we argue that $\square_1(\kappa)$ holds in $V[G*h_0*h_1]$. The sequence $$\bigcup h_0=\vec C=\<C_\alpha\st \alpha\in\inacc(\kappa)\>$$ is a $\square_1(\kappa)$-sequence in $V[G*h_0]$ by Theorem~\ref{theorem_reflection_characterization} because we can show that $\Refl_0(\kappa)$ holds by essentially the same argument as for $\Refl_1(\kappa)$ above. Suppose that $\vec{C}$ is no longer a $\square_1(\kappa)$-sequence in $V[G*h_0*h_1]$. This implies that there is a condition $p\in h_1$ such that in $V[G*h_0]$,
\begin{center}
$p\forces_{\Q_{\kappa,1}}$ ``there is a $1$-club $\dot{E}\subseteq\check{\kappa}$ that threads $\check{\vec{C}}$''.
\end{center}
Let $g_1$ be generic for $\Q_{\kappa,1}$ over $V[G*h_0*h_1]$ with $p\in g_1$. In $V[G*h_0*(h_1\times g_1)]$, let $E=\dot{E}_{h_1}$ and $E^*=\dot{E}_{g_1}$. By mutual genericity, we may fix $\alpha\in E\setminus E^*$. A proof almost identical to that of Claim \ref{claim_pi12_indescribable} shows that $\kappa$ is $\Pi^1_2$-indescribable in $V[G*h_0*(h_1\times g_1\times h_2)]$ and hence weakly compact in $V[G*h_0*(h_1\times g_1)]$. Now, in $V[G*h_0*(h_1\times g_1)]$, fix any $j:M\to N$  with critical point $\kappa$ and $E, E^*\in M$. Since both are 1-clubs, $\kappa\in j(E)\cap j(E^*)$, and so by elementarity there is an inaccessible $\beta\in\kappa\setminus(\alpha+1)$ such that $E\cap\beta$ and $E^*\cap\beta$ are both stationary in $\beta$. But then, as they both thread $\vec{C}$, it must be the case that $E\cap\beta=C_\beta=\hat{E}\cap\beta$. This contradicts the fact that $\alpha\in E\setminus E^*$ and finishes the proof of the theorem.
\end{proof}
\begin{remark}\label{remark_nonresurectionfails}
Observe that $\kappa$ cannot be $\Pi^1_2$-indescribable in $V[G*h_0*h_1]$ because $\square_1(\kappa)$ holds there. Thus, the set $S$, where $\Tr_1^\zeta([\kappa])^{V[G*h_0*h_1]}=[S]$, cannot be $\Pi^1_2$-indescribable in $V[G*h_0*h_1]$, which shows that Proposition~\ref{proposition_nonresurection} can fail for $\Pi^1_2$-indescribable sets.
\end{remark}
\section{An application to simultaneous reflection}\label{section_applications}
In this section we will show that the simultaneous reflection principle $\Refl_n(\kappa,2)$ is incompatible with $\square_n(\kappa)$.

\begin{theorem}\label{theorem_n_square_refutes_simultaneous_refl}
  Suppose that $1\leq n<\omega$, $\kappa$ is $\Pi^1_n$-indescribable and $\square_n(\kappa)$ holds. Then
  there are two $\Pi^1_n$-indescribable subsets $S_0, S_1 \subseteq \kappa$ that do not
  reflect simultaneously, i.e., there is no $\beta < \kappa$ such that
  $S_0 \cap \beta$ and $S_1 \cap \beta$ are both $\Pi^1_n$-indescribable subsets of $\beta$.
\end{theorem}

\begin{proof}
  Suppose for the sake of contradiction that every pair of $\Pi^1_n$-indescribable subsets
  of $\kappa$ reflects simultaneously. Already, $\Refl_n(\kappa)$ implies that $\kappa$ is $\omega$-$\Pi^1_n$-indescribable (see \cite{MR3985624}),
  so the set $E = \{\alpha < \kappa \st \Refl_{n-1}(\alpha) \mbox{ holds}\}$ is a $\Pi^1_n$-indescribable subset of $\kappa$ because the set of $\Pi^1_n$-indescribable cardinals below $\kappa$ is $\Pi^1_n$-indescribable and $(n-1)$-reflection holds at each of them.

  Let $\vec{C} = \langle C_\alpha \st \alpha\in\Tr_{n-1}(\kappa) \rangle$ be a $\square_n(\kappa)$-sequence. For all
  $\alpha \in \Tr_{n-1}(\kappa)$, let
  \begin{align*}
    S^0_\alpha &= \{\beta \in \Tr_{n-1}(\kappa) \setminus (\alpha + 1) \st
    C_\beta \cap \alpha \in\Pi^1_{n-1}(\alpha)^+\} \text{ and} \\
    S^1_\alpha &= \Tr_{n-1}(\kappa) \setminus ((\alpha + 1) \cup S^0_\alpha).
  \end{align*}
Let $A=\{\alpha\in\Tr_{n-1}(\kappa)\st S^0_\alpha\in \Pi^1_n(\kappa)^+\}$.

  \begin{claim}\label{claim_indescribable_index_set}
    $A$ is $\Pi^1_n$-indescribable in $\kappa$.
  \end{claim}

  \begin{proof}
Fix an $n$-club $C\subseteq\kappa$. Since $\Tr_{n-1}(C)$ is an $n$-club in $\kappa$, it follows that $E\cap \Tr_{n-1}(C)$ is $\Pi^1_n$-indescribable in $\kappa$. For each $\beta\in E\cap\Tr_{n-1}(C)$, $\Refl_{n-1}(\beta)$ holds and $C_\beta\cap C$ is a $\Pi^1_{n-1}$-indescribable subset of $\beta$. Thus, for $\beta\in E\cap \Tr_{n-1}(C)$, we may let $\alpha_\beta$ be the least $\Pi^1_{n-1}$-indescribable cardinal such that $C_\beta\cap C\cap\alpha_\beta$ is $\Pi^1_{n-1}$-indescribable in $\alpha_\beta$. Notice that $\alpha_\beta\in C$ for all $\beta\in E\cap \Tr_{n-1}(C)$ because $C$ is an $n$-club. Since the map $\beta\mapsto\alpha_\beta$ is regressive on $E\cap\Tr_{n-1}(C)$, it follows by the normality of $\Pi^1_n(\kappa)$ that there is a fixed $\alpha\in C$ and a $\Pi^1_n$-indescribable set $T\subseteq E\cap\Tr_{n-1}(C)$ such that $\alpha_\beta=\alpha$ for all $\beta\in T$. This implies that $T\subseteq S^0_\alpha$, and thus $\alpha\in A\cap C$.
  \end{proof}

  \begin{claim}
    There is $\alpha \in A$ such that $S^1_\alpha$ is a $\Pi^1_n$-indescribable subset of $\kappa$.
  \end{claim}

  \begin{proof}
    Suppose not, and let $\alpha_0 < \alpha_1$ be elements of $A$. Since $E$ is $\Pi^1_n$-indescribable in $\kappa$ and $S^1_{\alpha_0}$
    and $S^1_{\alpha_1}$ are both in the $\Pi^1_n$-indescribability ideal on $\kappa$, we can find
    $\beta \in E \setminus ((\alpha_1 + 1) \cup S^1_{\alpha_0} \cup S^1_{\alpha_1})$.
    It follows that $\beta \in S^0_{\alpha_0} \cap S^0_{\alpha_1}$, so, by the coherence
    properties of the $\square_n(\kappa)$-sequence, we have $C_\beta \cap \alpha_0 =
    C_{\alpha_0}$ and $C_\beta \cap \alpha_1 = C_{\alpha_1}$, and hence
    $C_{\alpha_1} \cap \alpha_0 = C_{\alpha_0}$. But then by Lemma \ref{lemma_indescribable_union} and Claim \ref{claim_indescribable_index_set}, we see that $\bigcup_{\alpha \in A} C_\alpha$ is a $\Pi^1_n$-indescribable subset of $\kappa$. Thus,
   $\bigcup_{\alpha \in A} C_\alpha$ is a thread through $\vec{C}$, which is a contradiction.
  \end{proof}

  We can therefore fix $\alpha \in \Tr_{n-1}(\kappa)$ such that both $S^0_\alpha$ and $S^1_\alpha$
  are $\Pi^1_n$-indescribable subsets of $\kappa$. Let $S_0 = S^0_\alpha$ and $S_1 = S^1_\alpha$.
  We claim that $S_0$ and $S_1$ cannot reflect simultaneously. Otherwise, there is
  $\gamma$ such that $S_0 \cap \gamma$ and $S_1 \cap \gamma$ are both $\Pi^1_n$-indescribable
  subsets of $\gamma$. Consider the $n$-club $C_\gamma$. Since $\gamma$ is $\Pi^1_n$-indescribable,
  $\mathrm{Tr}_{n-1}(C_\gamma)$ is also an $n$-club in $\gamma$. We can therefore find $\beta_0 < \beta_1$
  in $\mathrm{Tr}_{n-1}(C_\gamma)$ such that $\beta_0 \in S_0$ and $\beta_1 \in S_1$.
  But note that $C_{\beta_0} = C_\gamma \cap \beta_0$ and $C_{\beta_1} = C_\gamma \cap
  \beta_1$, so $C_{\beta_0} = C_{\beta_1} \cap \beta_0$, contradicting the fact
  that $C_{\beta_0} \cap \alpha$ is $\Pi^1_{n-1}$-indescribable in $\alpha$ whereas
  $C_{\beta_1} \cap \alpha$ is not $\Pi^1_{n-1}$-indescribable in $\alpha$.
\end{proof}

As a direct consequence of Theorem \ref{theorem_1_square_and_reflection} and Theorem \ref{theorem_n_square_refutes_simultaneous_refl} we obtain the following.

\begin{corollary}\label{corollary_reflection}
Suppose $\kappa$ is $\Pi^1_2$-indescribable. Then there is a forcing extension in which $\Refl_1(\kappa)$ and $\lnot\Refl_1(\kappa,2)$ both hold.
\end{corollary}

\section{Questions}\label{section_questions}
The theorems proved in this article about the principle $\square_1(\kappa)$ do not easily generalize to $\square_n(\kappa)$ because several key technical results about $\Pi^1_1$-indescribability which we used crucially in the proofs no longer hold for higher orders of indescribability. For example, given an embedding $j:M\to N$, where $N$ is $\Pi^1_n$-correct, we cannot necessarily use a generic $G$ for a poset $\P\in N$ from the ground model to lift  $j$ because $N[G]$ may no longer be $\Pi^1_n$-correct. An illustration of this is given in Remark~\ref{rem_groundmodelGenericBreaksCorrectness}. Also, while $\kappa$-strategically closed forcing cannot make a subset of $\kappa$ $\Pi^1_1$-indescribable if it was not so already in the ground model by Proposition~\ref{proposition_nonresurection}, a set can become $\Pi^1_2$-indescribable after $\kappa$-strategically closed forcing by Remark~\ref{remark_nonresurectionfails}.

\begin{question}
For $n>1$, can we force from a strong enough large cardinal that $\kappa$ is $\Pi^1_n$-indescribable and $\square_n(\kappa)$ holds nontrivially?
\end{question}
\begin{question}
Relative to large cardinals, for $n>1$, is it consistent that $\Refl_n(\kappa)$ and $\square_n(\kappa)$ both hold?
\end{question}

\begin{question}
Relative to large cardinals, is it consistent that $\Refl_1(\kappa,2)$ + $\lnot\Refl_1(\kappa,3)$?
\end{question}

\begin{question}
Can we force any indestructibility of $\Refl_1(\kappa)$?
\end{question}

\begin{question}
For $1\leq n<\omega$, if $\kappa$ is $\Pi^1_n$-indescribable, does our principle $\square_n(\kappa)$ imply the Brickhill-Welch principle $\square^n(\kappa)$? See Remark \ref{remark_relationship} for a discussion of $\square^n(\kappa)$.
\end{question}


\end{document}